\documentclass[11pt]{amsart}
\usepackage{amscd, amssymb}

\usepackage{amsthm}
\usepackage{mathrsfs}
\usepackage{bbm}
\usepackage{ stmaryrd }

\usepackage{multirow}
\usepackage{makecell}

\usepackage{hyperref}
\hypersetup{
	colorlinks=true,
	linkcolor=blue,
	filecolor=magenta,  
	urlcolor=cyan,
}

\def\Aut{\operatorname{Aut}}
\def\End{\operatorname{End}}

\def\ker{\operatorname{ker}}

\def\dim{\operatorname{dim}}
\def\ad{\operatorname{ad}}

\def\id{\operatorname{id}}
\def\Dist{\operatorname{Dist}}

\def\Ber{\operatorname{Ber}}
\def\ber{\operatorname{ber}}
\def\Ind{\operatorname{Ind}}
\def\Coind{\operatorname{Coind}}
\def\Ad{\operatorname{Ad}}

\def\Hom{\operatorname{Hom}}

\def\C{\mathbb{C}}

\def\N{\mathbb{N}}
\def\Z{\mathbb{Z}}

\def\AA{\mathcal{A}}
\def\TT{\mathcal{T}}

\def\OO{\mathcal{O}}

\def\UU{\mathcal{U}}

\def\FF{\mathcal{F}}
\def\DD{\mathcal{D}}

\def\ZZ{\mathcal{Z}}

\def\a{\mathfrak{a}}
\def\b{\mathfrak{b}}
\def\c{\mathfrak{c}}
\def\m{\mathfrak{m}}
\def\n{\mathfrak{n}}
\def\t{\mathfrak{t}}
\def\p{\mathfrak{p}}
\def\q{\mathfrak{q}}
\def\g{\mathfrak{g}}
\def\h{\mathfrak{h}}
\def\r{\mathfrak{r}}
\def\k{\mathfrak{k}}
\def\l{\mathfrak{l}}
\def\s{\mathfrak{s}}
\def\o{\mathfrak{o}}

\def\d{\partial}

\def\ol{\overline}

\def\tr{\text{tr}}

\def\Spec{\text{Spec}}

\def\sub{\subseteq}

\def\xto{\xrightarrow}

\newtheorem{thm}{Theorem}[section]
\newtheorem{cor}[thm]{Corollary}
\newtheorem{lemma}[thm]{Lemma}
\newtheorem{prop}[thm]{Proposition}
\newtheorem{conj}[thm]{Conjecture}

\theoremstyle{definition}
\newtheorem{definition}[thm]{Definition}
\theoremstyle{remark}
\newtheorem{remark}[thm]{Remark}

\numberwithin{equation}{section}

\usepackage{relsize}
\usepackage{fancyhdr}
\usepackage[all,cmtip]{xy}

\providecommand{\keywords}[1]
{
	\small	
	\textbf{\textit{Keywords---}} #1
}

\begin{document}
\title{Ghost distributions on supersymmetric spaces I: Koszul induced superspaces, branching,  and the full ghost centre}
\chead{Ghost distributions on supersymmetric spaces}

\author[Alex Sherman]{Alexander Sherman}

\maketitle
\pagestyle{plain}

\begin{abstract}  Given a Lie superalgebra $\g$, Gorelik defined the anticentre $\AA$ of its enveloping algebra, which consists of certain elements that square to the center.  We seek to generalize and enrich the anticentre to the context of supersymmetric pairs $(\mathfrak{g},\mathfrak{k})$, or more generally supersymmetric spaces $G/K$.  We define certain invariant distributions on $G/K$, which we call ghost distributions, and which in some sense are induced from invariant distributions on $G_0/K_0$.  Ghost distributions, and in particular their Harish-Chandra polynomials, give information about branching from $G$ to a symmetric subgroup $K'$ which is related (and sometimes conjugate) to $K$.  We discuss the case of $G\times G/G$ for an arbitrary quasireductive supergroup $G$, where our results prove the existence of a polynomial which determines projectivity of irreducible $G$-modules.  Finally, a generalization of Gorelik's ghost centre is defined which we call the full ghost centre, $\ZZ_{full}$.  For type I basic Lie superalgebras $\g$ we fully describe $\ZZ_{full}$, and prove that if $\g$ contains an internal grading operator, $\ZZ_{full}$ consists exactly of those elements in $\UU\g$ acting by $\Z$-graded constants on every finite-dimensional irreducible representation.
\end{abstract}



\section{Introduction}

Let $\g$ be a Lie superalgebra over an algebraically closed field $k$ of characteristic zero.  In \cite{gorelik2000ghost}, Gorelik defined a certain natural 'twisted' adjoint action of a Lie superalgebra $\g$ on its enveloping algebra $\UU\g$.  The action was originally considered by \cite{arnaudon1997casimir} for $\o\s\p(1|2n)$, where a certain element $T$ in the enveloping algebra was constructed, called Casimir's ghost, which squared to the center.

The action defined by Gorelik in general is remarkable in that the structure of $\UU\g$ becomes that of an induced module, $\Ind_{\g_{\ol{0}}}^{\g}\UU\g_{\ol{0}}$.  Further the invariants of this action, denoted $\AA$ and called the anticentre in \cite{gorelik2000ghost}, are both a module over the center $\ZZ:=\ZZ(\UU\g)$ and multiply into the center, so that $\tilde{\ZZ}:=\ZZ+\AA$ is an algebra which Gorelik called the ghost centre of $\g$.  If $\g$ is quasireductive, i.e. $\g_{\ol{0}}$ is reductive and acts semisimply on $\g$, and $\Lambda^{top}\g_{\ol{1}}$ is a trivial $\g_{\ol{0}}$-module, Gorelik obtained an explicit identification of vector spaces between $\AA$ and the center of $\UU\g_{\ol{0}}$.  Further, for basic classical Lie superalgebras, Gorelik gave a complete description of the Harish-Chandra image of $\AA$ and the structure of $\tilde{\ZZ}$.  In particular she showed that $\tilde{\ZZ}$ consists of exactly those elements of $\UU\g$ that act by superconstants on every irreducible representation.  We also note that Gorelik fully computed the $\AA$ for $Q$-type superalgebras in \cite{gorelik2006shapovalov}.

\subsection{Generalization to supersymmetric spaces} We reformulate Gorelik's results from the geometric perspective, and thereby understand how they may be generalized.  Consider the setting of symmetric supervarieties (or supersymmetric spaces, if one prefers that term) $G/K$ and their corresponding supersymmetric pairs $(\g,\k)$, corresponding to an involution $\theta$.  We have a decomposition $\g=\k\oplus\p$ into the $\pm1$ eigenspaces of $\theta$.  We define a new subalgebra 
\[
\k':=\k_{\ol{0}}\oplus\p_{\ol{1}},
\]
which itself is the fixed points of the involution $\delta\circ\theta$, where $\delta(x)=(-1)^{\ol{x}}x$ is the grading operator on $\g$.  Let $K'$ be the subgroup of $G$ with $K'_0=K_0$ and $\operatorname{Lie}(K')=\k'$.  Then because $\k_{\ol{1}}'\oplus\k_{\ol{1}}=\g_{\ol{1}}$, the action of $K'$ on $G/K$ enjoys many nice properties, in particular many of which are not explicitly enjoyed by the action of $K$ on $G/K$.  The first result indicating this is the following.  Recall that for an affine supervariety $X$ with a closed point $x$ and maximal ideal $\m_x\sub k[X]$ in the space of functions, we may consider the super vector space of distributions
\[
\Dist(X,x)=\{\psi:k[X]\to k:\psi(\m_x^n)=0\text{ for }n\gg0\}.
\]
If we consider the point $eK$ on $G/K$, $K_0'$ fixes it, and thus $K'$ has a natural action on $\Dist(G/K,eK)$.
\begin{thm}\label{dist_iso_intro}
	We have an isomorphism of $K'$-modules
	\[
	\Dist(G/K,eK)\cong\Ind_{\k_{\ol{0}}}^{\k}\Dist(G_0/K_0,eK_0).
	\]
	In particular, if $K'$ is quasireductive and $\Lambda^{top}\p_{\ol{1}}$ is a trivial $K_0$-module, then we have an explicit isomorphism of vector spaces
	\[
	\Dist(G/K,eK)^{K'}\cong \Dist(G_0/K_0,eK_0)^{K_0}.
	\]
\end{thm}
For the symmetric supervariety $G\times G/G\cong G$, we have identifications $G'\cong G$ and $\Dist(G,eG)\cong\UU\g$, and the action we obtain in this case of $G'$ on $\Dist(G,eG)$ is exactly Gorelik's twisted adjoint action, thus reproducing her results from this context.  

The proof of Theorem \ref{dist_iso_intro} uses a construction due to Koszul (\cite{koszul1982graded}) which allows one to take the variety $G_0/K_0$, which has an action of $K_0'=K_0$, and induce it to a $K'$-supervariety denoted $(G_0/K_0)^{(\k')}$.  The latter supervariety has the special property that the vector fields $\k_{\ol{1}}'=\p_{\ol{1}}$ are everywhere non-vanishing, and its algebra of functions and spaces of distributions, respectively, are (co)induced from $G_0/K_0$.  Then by a general result, $G/K$ and $(G_0/K_0)^{(\k')}$ are locally isomorphic as $K'$-supervarieties, implying an isomorphism of their spaces of distributions as in Theorem \ref{dist_iso_intro}.

\subsection{The Harish-Chandra morphism} Now consider a symmetric supervariety $G/K$ where we assume $G$ is quasireductive, $\Lambda^{top}\p_{\ol{1}}$ is the trivial $K_0$-module, and we have an Iwasawa decomposition $\g=\k\oplus\a\oplus\n$ where $\a\sub\p$ is a Cartan subspace (see Definition \ref{defn_cartan_subspace}).  We write 
\[
\AA_{G/K}:=\Dist(G/K,eK)^{K'},
\]
for the $K'$-invariant distributions, and call elements of $\AA_{G/K}$ ghost distributions on $G/K$.  One important problem is to compute the image of $\AA_{G/K}$ under the Harish-Chandra homomorphism 
\[
HC:\Dist(G/K,eK)\to \mathfrak{A},
\]
where $\mathfrak{A}$ is the polynomial superalgebra on $\a$.  Then we have will that
\[
HC(\AA_{G/K})\sub S(\a_{\ol{0}})\Lambda^{top}\a_{\ol{1}}.
\]
Thus given $\gamma\in\AA_{G/K}$, we may write $HC(\gamma)=p_{\gamma}\xi$, $p_{\gamma}\in S(\a_{\ol{0}})$ and $\xi\in\Lambda^{top}\a_{\ol{1}}$ is a fixed basis.  If $B$ is a Borel subgroup of $G$ such that $\a\oplus\n\sub\operatorname{Lie}(B)$, then an irreducible $B$-submodule $L_{\lambda}\sub\C[G/K]$ will be of weight $\lambda\in\a_{\ol{0}}^*$, and $\a_{\ol{1}}$ will act on $L_{\lambda}$ as an odd abelian Lie superalgebra. Then $HC(\gamma):L_{\lambda}\to k$ will define a functional.

\begin{thm}\label{intro_branching_thm}  Let $L_{\lambda}\sub\C[G/K]$ be an irreducible $B$-submodule; the following are equivalent:
	\begin{enumerate}
		\item  the $K'$-module generated by $L_{\lambda}$ contains a copy of $I_{K'}(k)$;
		\item  there exists $\gamma\in\AA_{G/K}$ such that the map $HC(\gamma):L_{\lambda}\to k$ is nonzero;
		\item  $L_{\lambda}$ is projective over $\a_{\ol{1}}$ and $p_{\gamma}(\lambda)$ for some $\gamma\in\AA_{G/K}$.
	\end{enumerate}
\end{thm}
Here $I_{K'}(k)$ denote the injective indecomposable $K'$-module with socle $k$.  From this we obtain, as one corollary:
\begin{cor}\label{intro_cor}
	Suppose that the equivalent conditions of Theorem \ref{intro_branching_thm} hold,and suppose further that the $G$-module generated by $f_{\lambda}$, $L$, is irreducible.	Then $I_{G}(L)$ is a submodule of $k[G/K']$.
\end{cor} 
Thus obtaining $HC(\AA_{G/K})$ is of interest, and this will be taken up in more detail for the case of basic classical Lie superalgebras in a subsequent article. However even for such Lie superalgebras the answer is not known in general (outside the case $G\times G/G$, originally done by Gorelik).  

\subsection{The case of $G\times G/G$} For $G\times G/G$, we have the following nice application of Corollary \ref{intro_cor}. 
\begin{thm}\label{intro_proj_criteria_poly}
	Let $G$ be a quasireductive supergroup with Cartan subalgebra $\h\sub\g$, and choose a Borel subgroup $B$  with Lie superalgebra $\b$ containing $\h$.  Then there exists a polynomial $p_{G,B}\in S(\h)$ of degree less than or equal to $\dim\b_{\ol{1}}$, such that for a $\b$-dominant weight $\lambda$, 
	\[
	p_{G,B}(\lambda)\neq0\ \text{ if and only if }\ L_{B}(\lambda)\text{ is projective},
	\]
where $L_{B}(\lambda)$ is the irreducible $G$-module of $B$-highest weight $\lambda$.  
\end{thm}
In particular the above result implies that if one simple $G$-module is projective then this is a generic property of simple $G$-modules.  However it is possible that $p_{G,B}$ is the zero polynomial, so that no simple $G$-modules are projective. This was already know for many Lie superalgebras by direct study, and our theory gives a general explanation for this phenomenon.

For a different application of Theorem \ref{intro_proj_criteria_poly}, we can prove the following general sufficient criteria for having projective irreducible modules.  We say a quasireductive supergroup is Cartan-even if a maximal torus of $G_0$ is self-centralizing.

\begin{thm}
	Let $G$ be a Cartan-even quasireductive supergroup with a chosen Cartan subalgebra $\h$ such that the following conditions hold on its Lie superalgebra:
	\begin{enumerate}
		\item for a root $\alpha\in\h^*$, $\dim[\g_{\alpha},\g_{-\alpha}]=1$; and
		\item for a root $\alpha\in\h^*$, the pairing
		\[
		[-,-]:\g_{\alpha}\otimes\g_{-\alpha}\to [\g_{\alpha},\g_{-\alpha}]
		\]
		is nondegenerate.
	\end{enumerate}
	Choose an arbitrary Borel subgroup $B$ whose Lie superalgebra contains $\h$, and let $\alpha_1,\dots,\alpha_n$ denote the odd positive roots with $r_i=\dim\g_{\alpha_i}$.  Write $h_{\alpha_i}$ for a nonzero element of $[\g_{\alpha_i},\g_{-\alpha_i}]$.  Then we have (up to a scalar)
	\[
	p_{G,B}=h_{\alpha_1}^{r_1}\cdots h_{\alpha_n}^{r_n}+l.o.t.
	\]
	In particular $p_{G,B}\neq0$, so $G$ admits irreducible projective modules.
\end{thm}
Note that the above conditions hold in particular for a Cartan-even quadratic quasireductive supergroup, i.e. a Cartan-even quasireductive supergroup whose Lie superalgebra admits a nondegenerate, invariant, and even supersymmetric form.  See \cite{benayadi2000quadratic} for a classification of quadratic quasireductive Lie superalgebras.

\subsection{The operator $T_{\g}$} Recall that $\AA$ denotes the ghost center of $\UU\g$.  Suppose that $\Lambda^{top}\g_{\ol{1}}$ is the trivial $\g_{\ol{0}}$-module.  Then $\AA$ contains an element of least degree, which we write as $T_{\g}$.  This operator can test whether a given semisimple $G$-module is projective in the following sense:

\begin{prop}
	Let $L$ be a simple $G$-module.  Then $T_{\g}$ acts by $0$ on $L$ if and only if $L$ is not projective.  If $L$ is projective, then $T_{\g}$ acts by one of the following automorphisms:
	\begin{itemize}
		\item if $T_{\g}$ is even, then up to scalar it acts by $\delta_{L}$, the parity operator on $L$;
		\item if $T_{\g}$ is odd, then it acts by $\delta_{L}\circ\sigma_{L}$, where $\sigma_L:L\to L$ is a $G$-equivariant odd automorphism of $L$.
	\end{itemize}
Further $T_{\g}^2$ is an even, central element of $\UU\g$ which exactly annihilates all non-projective simple $G$-modules.
\end{prop}
Note that this operator was used to understand the representation theory of $\o\s\p(1|2n)$ in \cite{arnaudon1997casimir} and \cite{musson1997center}.  A conceptual description of $T_{\g}$ was given in \cite{duflo2007symmetric}, where they give a universal formula for $T_{\g}$ in terms of the Jacobian of a certain exponential map.

We conjecture a stronger property of $T_{\g}$; first of all, let us remove the condition that ${\bigwedge}^{top}\g_{\ol{1}}$ is a trivial $\g_{\ol{0}}$-module.  Then $\UU\g$ admits a least degree element $T_{\g}$ under the twisted adjoint action is an eigenvector of weight the same as that of ${\bigwedge}^{top}\g_{\ol{1}}$.  Even though this operator $T_{\g}$ is not $\g$-invariant, however it still always takes $\g$-submodules to $\g$-submodules.  

\begin{conj}\label{conj on T}
	Let $L$ be a simple $\g$-module.  Then $T_{\g}$ acts on $P(L)$ by taking the top of $P(L)$ to its socle; in particular it annihilates the radical of $P(L)$.
\end{conj}
If Conjecture \ref{conj on T} holds, it would in particular tell us how $T_{\g}$ acts on the \textit{entire} category $\operatorname{Rep}G$.  We note that the this conjecture is straightfoward to show when $G_0$ is a central torus, and this case is proven in \cite{GorSerSher2022}.

\subsection{The full ghost centre $\ZZ_{full}$}  We may generalize Gorelik's results in a different way as follows to produce an interesting subalgebra of $\UU\g$ which contains Gorelik's ghost centre. Let $\Aut(\g,\g_{\ol{0}})$ be those automorphisms of $\g$ that fix $\g_{\ol{0}}$ pointwise.  For $\phi\in\Aut(\g,\g_{\ol{0}})$, define the $\phi$-twisted adjoint action $\ad_{\phi}$ of $\g$ on $\UU\g$ by
\[
\ad_{\phi}(u)(v)=uv-(-1)^{\ol{u}\ol{v}}v\phi(u).
\]
When $\phi=\delta$, we obtain the twisted adjoint action studied by Gorelik.  Then using Theorem \ref{dist_iso_intro} or a similar method to that used in \cite{gorelik2000ghost}, we can prove that:
\begin{thm}\label{intro_thm_induced_module_twisted_adjoint}
	If $\phi(x)=x$ implies that $x\in\g_{\ol{0}}$, then $\UU\g\cong\Ind_{\g_{\ol{0}}}^{\g}\UU\g_{\ol{0}}$ under the $\phi$-twisted adjoint action.  
\end{thm}
Write $\AA_{\phi}\sub\UU\g$ for the $\ad_{\phi}$-invariant elements in $\UU\g$, for any $\phi\in\Aut(\g,\g_{\ol{0}})$.  Then $\AA_{\id}=\ZZ$, and $\AA_{\delta}=\AA$.  Further, for $\phi,\psi\in\Aut(\g,\g_{\ol{0}})$ multiplication induces a morphism
\[
\AA_{\phi}\otimes\AA_{\psi}\to\AA_{\psi\phi}.
\]
Therefore if we set
\[
\ZZ_{full}:=\sum\limits_{\phi\in\Aut(\g,\g_{\ol{0}})}\AA_{\phi},
\]
we obtain a subalgebra of $\UU\g$, which also contains $\tilde{\ZZ}$.  For $\g$ one of the type I basic classical Lie superalgebras $\g\l(m|n)$, $\s\l(m|n)$, $\p\s\l(n|n)$ with $n>2$, or $\o\s\p(2|2n)$) we have an explicit description of this algebra.  Here, $\Aut(\g,\g_{\ol{0}})\cong k^\times$, so we write $\phi_c\in\Aut(\g,\g_{\ol{0}})$ for the automorphism corresponding to $c\in k^\times$ and $\AA_{c}$ for the invariants of the $\phi_c$-twisted adjoint action.  Then $\phi_c$ satisfies the conditions of Theorem \ref{intro_thm_induced_module_twisted_adjoint} exactly if $c\neq1$.  
\begin{thm}
	Let $N=\dim\g_{\ol{1}}/2$.  Then $HC(\AA_{c})=HC(\AA_{-1})$ for all $c\neq 1$, and 
	\[
	\ZZ_{full}=\bigoplus\limits_{\zeta^N=1}\AA_{\zeta}.
	\]
	Further, for $\g\l(m|n)$, $\s\l(m|n)$ with $m\neq n$, and $\o\s\p(2|2n)$, $\ZZ_{full}$ consists of exactly the set of elements in $\UU\g$ which act by $\Z$-graded constants on all finite-dimensional irreducible representations of $\g$.
\end{thm}
We exclude $\s\l(n|n)$ and $\p\s\l(n|n)$ in the last statement due to the lack of an internal grading operator.  Further, $\p\s\l(2|2)$ is fully excluded because in this case $\Aut(\g,\g_{\ol{0}})=SL_2$ (see section 5.5 of \cite{musson2012lie}).  For this reason it would be interesting to compute the full ghost center for this superalgebra.

\subsection{Future work} This article is the first of two on ghost distributions and their applications.  In the subsequent article we study in more detail two questions of interest: (1) when is it possible to produce an algebra using ghost distributions on a general supersymmetric space?  In particular can we form an algebra of differential operators, and if not can we at least produce an algebra of polynomials using the Harish-Chandra homomorphism?  (2) Computing $HC(\AA_{G/K})$ as much as possible in the case when $\g$ is an almost simple basic classical Lie superalgebra and the involution under consideration preserves the form on $\g$.  In this case we give some general properties of $HC(\AA_{G/K})$, and seek to compute it for all rank one supersymmetric pairs.

\subsection{Outline of paper}
In section 2 we introduce the basic algebraic supergeometry we need, in particular the algebra of differential operators and the space of distributions at a given point. Section 3 recalls basic facts about algebraic supergroups and their actions, and gives a description of invariant differential operators on homogeneous spaces.  In section 4, we explain the main technical construction, the induced superspace as defined by Koszul.  Section 5 applies the ideas of section 4 to homogeneous superspaces, and deduces what we need to generalize the results of \cite{gorelik2000ghost}.  Section 6 studies the case of $G/G_0$, looking in particular at a certain invariant distribution, $v_{\g}$, which will play an important role in the theory of ghost distributions.  Section 7 looks at applications to a symmetric supervariety $G/K$, and section 8 studies more closely the case when an Iwasawa decomposition is present, giving the definition of the Harish-Chandra map and its interpretations.  In section 9 we take a special look at the case of $G\times G/G$, where the theory of ghost distributions is most developed, and prove Theorem \ref{intro_proj_criteria_poly}.  Finally in section 10 we define and study the full ghost centre $\ZZ_{full}$, and look especially at the cases of type I algebras.   

\subsection{Acknowledgments}  The author would like to thank Alexander Alldridge for many enlightening discussions about this project.  Thanks to Vera Serganova for her patience, tremendous support and guidance throughout my PhD, during which much of this work was done.  Thank you to Maria Gorelik for patiently explaining many aspects of her relevant papers to me.   Finally thank you to Siddartha Sahi, Johannes Flake, and Inna Entova-Aizenbud for helpful comments and suggestions.  This research was partially supported by ISF grant 711/18 and NSF-BSF grant 2019694.

\section{Preliminaries from algebraic supergeometry}

\subsection{Linear super algebra notation} We work throughout over an algebraically closed field $k$ of characteristic zero.  For a super vector space $V$ we write $V=V_{\ol{0}}\oplus V_{\ol{1}}$ for its parity decomposition.  Even though we precisely consider $V_{\ol{0}}$ and $V_{\ol{1}}$ to be even vector spaces, we will occasionally (and abusively) view them as super vector spaces where $V_{\ol{0}}$ is purely even and $V_{\ol{1}}$ is purely odd. We write $\dim V=\dim V_{\ol{0}}+\dim V_{\ol{1}}$ for the dimension of the underlying vector space of $V$, and the superdimension of $V$ is given by $\operatorname{sdim}V=\dim V_{\ol{0}}-\dim V_{\ol{1}}$.

\subsection{Algebraic supergeometry notation} We will use the symbols $X,Y,\dots$ for supervarieties with even subschemes $X_0,Y_0,\dots.$ We will be considering supervarieties in the sense of \cite{sherman2019sphericalsupervar}, however all spaces of interest will be smooth and affine.  A smooth affine supervariety is always given by ${\bigwedge}^\bullet E$, the exterior algebra of a vector bundle $E$ on a smooth affine variety $X_0$ (see \cite{voronov1990elements}).  Thus one will lose almost nothing if one simply works with supervarieties of this form in this article.  Note that affine supervarieties and morphisms between them are entirely determined by their spaces of global functions and maps between them, just as with affine varieties.  See \cite{carmeli2011mathematical} for more on the basics of algebraic supergeometry.

If $X$ is a supervariety, there is a canonical closed embedding $i_X:X_0\to X$, and this is a homeomorphism of underlying topological spaces.  The closed points of $X$ are the $k$-points, which we write as $X(k)$, and they are canonically identified with the closed points of $X_0$ via $i_X$.  If $x$ is a closed point of $X$ and $\FF$ is a sheaf on $X$, we write $\FF_x$ for the stalk of $\FF$ at $x$.  Then $\OO_{X,x}$ is a local superalgebra, and we write $\m_x$ for its corresponding maximal ideal and $_0\m_x$ for the maximal ideal in $\OO_{X_0,x}$.  For affine supervarieties we will also write, by abuse of notation, $\m_x$ for the maximal ideal of $k[X]$ corresponding to $x$, and similarly $_0\m_x$ for the maximal ideal of $k[X_0]$ corresponding to $x$.

\subsection{Differential operators and distributions}
\begin{definition}
	For a supervariety $X$, let $\DD_X$ denote the sheaf of filtered algebras which is the subsheaf of $\mathcal{E}nd(\OO_X)$ defined inductively as follows.  We set $\DD_X^{n}=0$ for $n<0$, and for $n\geq 0$ and an open subset $U$ of $X$, set
	\[
	\Gamma(U,\DD_X^n):=\{D\in\End(\OO_U)):[D,f]\in\Gamma(U,\DD_X^{n-1})\text{ for all }f\in\OO_X(U)\}
	\]
We call $\DD_X$ the sheaf of differential operators on $X$, and refer to its sections as differential operators.  If $\FF$ is a sheaf on $X$, we say that it is a left, resp. right $\DD_X$-module if it is exactly that.
\end{definition}

The study of differential operators and the modules over them was initiated by Penkov in \cite{penkov1983d}.

\begin{definition}\label{dist_def}
	 If $x$ is a closed point of $X$, define the super vector space of distributions at $x$ to be all (not necessarily even) linear maps $\psi:\OO_{X,x}\to k$ such that for some $n\in\N$ we have $\psi(\m_x^n)=0$.  We denote this super vector space by $\Dist(X,x)$.  Define $\Dist^n(X,x)\sub\Dist(X,x)$ to be those distributions vanishing on $\m_x^{n+1}$ so that $\Dist(X,x)$ obtains a filtration.  Note that $\Dist^0(X,x)$ is one-dimensional and consists of the distinguished even distribution given by evaluation at $x$, which we denote by $\operatorname{ev}_x$. 
	
       We may give $\Dist(X,x)$ the structure of a right $\DD_X$-module as follows.  We view $\Dist(X,x)$ as a sheaf on $X$ supported on $x$.  Given a differential operator $D$ defined in a neighborhood of $x$, and a distribution $\psi$, define
	\[
	(\psi D)(f):=\psi(Df)
	\]
	This action respects the filtration on $\Dist(X,x)$, so it becomes a filtered right $\DD_X$-module.
\end{definition}

The following lemma is proved in the same way as in the classical setting, so we omit the proof.
\begin{lemma}Let $X$ be a supervariety with a closed point $x$.  \begin{enumerate}	
		\item Given a map of supervarieties $\phi:X\to Y$, we have a natural map of filtered super vector spaces $d\phi_x:\Dist(X,x)\to\Dist(Y,\phi(x))$. 
		\item The chain rule holds: if $\phi:X\to Y$ and $\psi:Y\to Z$, then $d(\psi\circ\phi)=d\psi\circ d\phi$.
		\item If $X$ is affine, then the natural pairing $\Dist(X,x)\otimes k[X]\to k$ has the property that if $\psi(f)=0$ for all $f\in k[X]$, then $\psi=0$.
		\item There is a natural restriction morphism $\operatorname{res}_x:\Gamma(U,\DD_X)\to\Dist(X,x)$ for any open subscheme $U$ containing $x$, given by $\operatorname{res}_x(D)(f)=D(f)(x)$.  This is a morphism of filtered right $\DD_X$-modules, where $\DD_X$ acts on itself by right multiplication.
	\end{enumerate}
\end{lemma}

%

\begin{remark}
	We have the following identifications:
	\[
	\Dist^n(X,x)=(\OO_{X,x}/\m_{x}^{n+1})^*, \ \ \Dist(X,x)=\lim\limits_{\rightarrow}(\OO_{X,x}/\m_{x}^{n+1})^*.
	\]
	Furthermore when $X$ is affine, we have:
	\[
	\Dist^n(X,x)=( k[X]/\m_x^{n+1})^*, \ \ \Dist(X,x)=\lim\limits_{\rightarrow}( k[X]/\m_x^{n+1})^*.
	\]
	In general we have an isomorphism of $\DD_{X}$-modules
	\[
	\Dist(X,x)=\Gamma_{\m_x}(\OO_{X,x})^*
	\]
	and for $X$ affine an isomorphism of $\Gamma(X,\DD_X)$-modules
	\[
	\Dist(X,x)=\Gamma_{\m_x} k[X]^*.
	\]
	We recall the definition of the functor $\Gamma_{\m_x}$:
	\[
	\Gamma_{\m_x}M=\{m\in M:\m_x^nm=0\text{ for }n\gg0\}.
	\]
\end{remark}

\begin{definition}
	Define the sheaf $\TT_X$ of vector fields on $X$ by setting $\Gamma(\Spec A,\TT_X)=\operatorname{Der}(A)$ for an affine open subscheme $\Spec A$ of $X$.  In this way $\TT_X$ becomes a subsheaf of $\DD_X^1$, and a sheaf of Lie superalgebras under supercommutator.
\end{definition}

\begin{definition}
	For a supervariety $X$ and a closed point $x$ of $X$, we define the tangent space of $X$ at $x$ to be the super vector space $T_xX:=(\m_x/\m_x^2)^*$.  In this way $T_xX$ is exactly the subspace of $\Dist^1(X,x)$ given by functionals $\psi:\OO_{X,x}/\m_x^2\to k$ such that $\psi(1)=0$.  
\end{definition}

	Observe the restriction morphism $\operatorname{res}_x:\DD_{X,x}\to\Dist(X,x)$ restricts to a morphism $\TT_{X,x}\to T_xX$.  
	
	\begin{definition}
		A supervariety $X$ is smooth if for all $x\in X(k)$ the morphism $\operatorname{res}_x:\TT_{X,x}\to T_xX$ is surjective.
	\end{definition}
	See the appendix of \cite{sherman2019sphericalsupervar} for more equivalent conditions of smoothness.

	We have the following standard theorem for $\DD_X$-modules which we will use later on.  The proof is almost verbatim from the classical case, so we omit the proof.

\begin{prop}\label{coherent_D_module_free}
	Suppose that $X$ is a smooth supervariety.  Then if a left (or right) $\DD_X$-module $\FF$ is coherent over $\OO_X$, then it is locally free over $\OO_X$.
\end{prop}
%
%
%

\begin{lemma}\label{dist_symm_lemma}
	Let $X$ be a supervariety and $x\in X(k)$ a closed point at which $X$ is smooth.  Suppose $V$ is a subspace of vector fields defined in a neighborhood $U$ of $x$, such that the restriction map $V\to T_xX$ is an isomorphism.  If $v_1,\dots,v_n,w_1,\dots,w_m$, where $\ol{v_i}=\ol{0}$ and $\ol{w_i}=\ol{1}$, is a homogeneous basis of $V$, then the restriction of the set of all monomials in $v_1,\dots,v_n,w_1,\dots,w_m$, in any order, is a spanning set of $\Dist(X,x)$
\end{lemma}
\begin{proof}
	Choose homogeneous functions $f_1,\dots,f_n,\alpha_1,\dots,\alpha_m$, where $\ol{f_i}=\ol{0}$ and $\ol{\alpha_i}=\ol{1}$, defined in a neighborhood of $x$ which vanish at $x$ such that they project to a basis of $\m_x/\m_x^2$.  Then by smoothness $\OO_{X,x}/\m_x^\ell$ is isomorphic to the set of superpolynomials in $f_1,\dots,f_n,\alpha_1,\dots,\alpha_m$ of degree less than or equal to $\ell$.  By applying a linear automorphism, we may assume that $\operatorname{res}_x(v_i)(f_j)=\delta_{ij}$ and $\operatorname{res}(w_i)(\alpha_j)=\delta_{ij}$.  Then we see that
	\[
	\operatorname{res}_x(v_1^{r_1}\cdots v_n^{r_n}w_1^{\epsilon_1}\cdots w_{m}^{\epsilon_m})(f_1^{s_1}\cdots f_n^{s_n}\alpha_1^{\delta_1}\cdots\alpha_{m}^{\delta_m})=\pm s_1!\cdots s_n!\delta_{r_1s_1}\cdots\delta_{r_ns_n}\delta_{\epsilon_1\delta_1}\cdots\delta_{\epsilon_m\delta_m}
	\]
	where $\epsilon_i,\delta_i\in\{0,1\}$, and the sign is determined by the number of parity changes which occurs in the computation.  From this the statement follows.
\end{proof}

\begin{definition}
	Suppose $\FF$ is a quasi-coherent sheaf on $X$.  For a closed point $x$ of $X$, define distributions on $\FF$ at $x$ to be all linear maps $\psi:\FF_{x}\to k$ such that for some $n\in\N$, $\psi(\m_x^n\FF_x)=0$.  We write this as $\Dist(\FF,x)$.  
	
	
\end{definition}

\begin{remark}  Notice that if $\FF$ has a connection, hence an action by vector fields, then $\Dist(\FF,x)$ naturally admits a right action by vector fields, given by
	\[
	(\psi v)(s)=\psi(vs).
	\]
   If the connection is flat and $X$ is smooth, then this action extends in the natural way to a right action by all differential operators.  
	
	As before, we have an identification for each $x\in X(k)$:
	\[
	\Dist(\FF,x)\cong\Gamma_{\m_x}(\FF_x)^*
	\]
	and this identification respects the right action by vector fields when $\FF$ admits a connection.
\end{remark}

\section{Differential Operators on a $G$-variety}

\subsection{Preliminaries on algebraic supergroups}  Recall that an algebraic supergroup is a group object in the category of algebraic supervarieties.  Morphisms of supergroups are those which respect the multiplication morphisms.  We only consider affine algebraic supergroups, or equivalently those algebraic supergroups that are linear, i.e.\ have a faithful finite-dimensional representation.  The letters $G,H,\dots$ will denote an affine algebraic supergroup.  To avoid cumbersome language, we will not write the adjectives affine and algebraic when referring to affine algebraic supergroups, and instead simply call them supergroups.  Similarly, we use the term subgroup instead of subsupergroup.  We refer again to \cite{carmeli2011mathematical} for basics on algebraic supergroups.

A supergroup has a Lie superalgebra which we will denote with the letters $\g,\h,\dots$.  The Lie superalgebra may be defined as the super vector space of left-invariant (or right-invariant) vector fields on $G$.  As in the classical case, the Lie superalgebra is canonically identified with the tangent space of $G$ at the identity, $T_eG$.  Morphisms of algebraic supergroups induce morphisms of the corresponding Lie superalgebras.

If $G$ is a supergroup, then $G_0$ is an algebraic group, and the morphism $i_{G}:G_0\to G$ is a morphisms of supergroups.  Further, $i_G$ induces an isomorphism of Lie algebras $\operatorname{Lie}(G_0)\cong \operatorname{Lie}(G)_{\ol{0}}$.

\subsection{Representations of supergroups} In this article we will be using some basic facts about the representation theory of quasireductive supergroups.  We refer to \cite{serganova2011quasireductive} for further background.

Recall that for an affine algebraic supergroup $G$, $k[G]$ has the natural structure of a supercommutative Hopf superalgebra.  We define a left $G$-module to be a left $k[G]$-comodule, and a right $G$-module to be a right $k[G]$-comodule.  It will be necessary for us to consider both left and right $G$-modules due to the fact that distributions form a right module over the algebra of differential operators, as we have seen.  The category of left $G$-modules is equivalent to the category of right $G$-modules because $k[G]$ has a Hopf structure.  We will sometimes call a left or right $G$-module simply a representation of $G$, or even a $G$-module, without specifying whether it has a left or right action.  In this case either the type of action is apparent or is of no importance.

A left (resp. right) $G$-module induces in a natural way a left (resp. right) representation of the Lie superalgebra.  Recall that the category of representations of $G$ is equivalent to the category of $(G_0,\g=\operatorname{Lie}(G))$-modules such that the action of $G_0$ and $\g_{\ol{0}}\sub\g$ are compatible (see \cite{carmeli2011mathematical}).

Given a representation $V$ of a Lie supergroup $G$  such that $\dim V_{\ol{1}}=n$, we write $\Ber(V)$ for the Berezinian of $V$, which is the one-dimensional $G$-module with the same parity as $n$, where the action by $G$ is given by the Berezinian morphism $G\to GL(V)\xto{\Ber}\mathbb{G}_m$ (see chapter 3 of \cite{manin2013gauge}).  If $V$ is purely even (resp. purely odd) then $\Ber(V)$ coincides with the top exterior power (resp. top symmetric power) of $V$.  We write $\ber_V:G\to \mathbb{G}_m$ for the character of $G$ determined by $\Ber(V)$, and by abuse of notation we also write $\ber_V$ for the character of $\g$ that this determines.  Then we have $\ber_V=\det_{V_{\ol{0}}}\cdot\det_{V_{\ol{1}}}^{-1}$ as a character of $G_0$, and $\ber_{V}=\tr_{V_{\ol{0}}}-\tr_{V_{\ol{1}}}$ as a character of $\g_{\ol{0}}$.

If $\chi:G\to\mathbb{G}_m$ is a character of $G$, and $V$ is a representation of $G$, we will write $V^{\chi}$ for the subspace of $V$ where $G$ acts by $\chi$.  If $\chi$ is the trivial character, we just write $V^G:=V^{\chi}$.

If $V$ is a $G$-module and $W\sub V$ is a subspace, the $G$-module generated by $W$, which we write as $\langle G\cdot W\rangle$, is given by $\UU\g\cdot\langle G_0\cdot W\rangle$. That is, we first take the $G_0$-module generated by $W$, and then take the $\UU\g$-module which that generates. 

Finally, if $V$ is a $G$-module then we will write (when they exist) $I_G(V)$, resp. $P_G(V)$ (or $I(V)$, resp. $P(V)$ when the context is clear) for the injective hull, resp projective cover of $V$.

\subsection{Actions of supergroups} If $X$ is a supervariety, $G$ a supergroup, and $G$ acts on (the left on) $X$, then we call $X$ a $G$-supervariety.  We will usually reserve the letter $a_X=a$ for the action morphism, i.e.\ $a:G\times X\to X$.   In this case we will consider $k[X]$ as a right $G$-module via translation.  Explicitly, $g\in G_0(k)$ acts by pullback $L_g^*$ along the left translation morphism $L_g:X\to X$.  The Lie superalgebra acts, for $u\in T_eG$, by
\[
u\mapsto (u\otimes 1)\circ a^*.
\]
This induces an map of superalgebras $\UU\g\to \Gamma(X,\DD_X)^{op}$, and in this way $\Dist(X,x)$ becomes a left $\UU\g$-module for any $x\in X(k)$.  In general, we will say $\g$ acts on a supervariety $X$ if it admits a homomorphism of algebras $\UU\g\to\Gamma(X,\DD_X)^{op}$ such that $\g$ maps into $\Gamma(X,\TT_X)$.
\begin{remark}\label{Dist_G_module}
	Suppose that $G$ acts on $X$, and $x\in X(k)$ is a closed point which is fixed by $G_0$.  Then $G_0$ acts on $\Dist(X,x)$.  However $\g$ also acts on $\Dist(X,x)$, and in this way $\Dist(X,x)$ obtains the structure of a $G$-module.  Notice that this will happen even if $x$ is not stabilized by all of $G$. 
\end{remark}

\subsection{Differential operators on a $G$-supervariety}    Let $X$ be an affine $G$-supervariety, with action morphism $a:G\times X\to X$, and consider $D_{X}=\Gamma(X,\DD_{X})$.  
\begin{lemma}
	For a differential operator $D\in D_{X}$, the following are equivalent:
	\begin{enumerate}
		\item The map $D: k[X]\to k[X]$ is $G$-equivariant;
		\item We have $a^*\circ D=\id\otimes D\circ a^*$;
	\end{enumerate}
and in the case when $G$ is connected, we have the third equivalent condition:
\begin{itemize}
	\item[(3)] For all $u\in\g$, $[u,D]=0$.
\end{itemize}
	In this case we say that $D$ is $G$-invariant.
\end{lemma}

\begin{proof}
	(1)$\iff$(3) is obvious when $G$ is connected.  And (2) says that $D$ is a $ k[G]$-comodule homomorphism, equivalently a $G$-module homomorphism giving (1)$\iff$(2).
\end{proof}
\begin{definition}
	Write $D_X^G$ for the superalgebra of $G$-invariant differential operators on $X$. 
\end{definition}
 
 For the meaning of an open orbit of an algebraic supergroup, see \cite{sherman2019sphericalsupervar}.

\begin{prop} Let $X$ be a $G$-supervariety, and $x$ a point of $X$ with stabilizer subgroup $K\sub G$.  If $G$ has an open orbit at $x$, then the morphism $\operatorname{res}_x:D_{X}\to\Dist(X,x)$ restricts to an injection $D_X^G\to\Dist(X,x)^K$.  
\end{prop}

\begin{proof}
	The map $\operatorname{res}_x$ is $K$-equivariant, so we have $\operatorname{res}_x(D_X^G)\sub\Dist(X,x)^K$.  To see that it is injective, let $a_x:G\to X$ be the orbit map at $x$, and $D\in D_X^G$.  Then by $a_x^*:\OO_X\to (a_x)_*\OO_G$ is an injective morphism of sheaves by Prop. 3.11 of \cite{sherman2019sphericalsupervar}.  Therefore, if $\operatorname{res}_x(D)=0$, we have $D(f)(x)=0$ for all $f$ defined in an open neighborhood of $x$.  Equivalently, $a_x^*(D(f))(e)=0$ for all such $f$.  But we have the factorization $a_x=a\circ (\id_G\times i_x)$, so this says that
	\begin{eqnarray*}
	(\id_G\times i_x)^*\circ a^*(D(f))& = &(\id_G\times i_x)^*(\id\otimes D)(a^*(f))\\
	                                  & = &(\id\otimes \operatorname{res}_x(D))(a^*(f))=0.
	\end{eqnarray*}
	This implies $D(f)=0$ for all $f$ defined in an open neighborhood of $X$.  Since by definition the restriction morphism on functions is injective for a supervariety, this implies that $D=0$.
\end{proof}
\begin{prop}\label{diff_ops_G/K}
	In the context of the previous proposition, if $X\cong G/K$ via the orbit map at $x$ then the map $D_X^G\to\Dist(X,x)^K$ is an isomorphism.
\end{prop}

\begin{proof}
	It remains to show it is surjective. For this, if $\psi\in\Dist(X,x)^K$, define $D_{\psi}$ by $f\mapsto (\id\otimes \psi)\circ a^*(f)$.  Indeed,
	\begin{eqnarray*}
	a^*\circ D_{\psi}& = &a^*\circ (\id\otimes \psi)\circ a^*\\
	                 & = &(\id\otimes\id\otimes\psi)\circ(a^*\otimes\id)\circ a^*\\
	                 & = &(\id\otimes\id\otimes\psi)\circ(\id\otimes a^*)\circ a^*\\
	                 & = &[\id \otimes((\id\otimes\psi)\circ a^*)]\circ a^*\\
	                 & = &(\id\otimes D_{\psi})\circ a^*.
	\end{eqnarray*}
\end{proof}

\section{Induced supervarieties in the sense of Koszul}

We present a construction that is originally due to Koszul in \cite{koszul1982graded}.  Although it may be defined for supervarieties that are not affine, for simplicity we stick to the affine case since that is all we need. Throughout, we let $H$ be a supergroup and write $\h=\operatorname{Lie}(H)$. 

\subsection{Induced and coinduced modules}

\begin{definition}
	Let $V_0$ be an $\h_{\ol{0}}$-module, and define the $\h$-module $\Ind_{\h_{\ol{0}}}^{\h}V_0$ to be 
	\[
	\Ind_{\h_{\ol{0}}}^{\h}V_0:=\UU\h\otimes_{\UU\h_{\ol{0}}}V_0.
	\]
	The action by $\h$ is left multiplication.  If the action of $\h_{\ol{0}}$ on $V_0$ integrates to an action of $H_0$, then the $\h$ action on $\Ind_{\h_{\ol{0}}}^{\h}V_0$ integrates to an action of $H$, where $H_0$ acts by
	\[
	h\cdot (u\otimes v)=\Ad(h)(u)\otimes h\cdot v.
	\]
	Similarly we define the $\h$-module $\Coind_{\h_{\ol{0}}}^{\h}V_0$ by
	\[
	\Coind_{\h_{\ol{0}}}^{\h}V_0:=\Hom_{\h_{\ol{0}}}(\UU\h,V_0).
	\]
	Here we consider $\UU\h$ as a left $\UU\h_{\ol{0}}$-module.  The action by $\h$ on this module is
	\[
	(u\eta)(v)=(-1)^{\ol{u}(\ol{\eta}+\ol{v})}\eta(vu).
	\]
	Once again, if $V_0$ is actually an $H_0$-module then $\Coind_{\h_{\ol{0}}}^{\h}V_0$ will be an $H$-module, where the action by $H_0$ is given by
	\[
	(h\cdot\eta)(v)=h\cdot\eta(\Ad(h^{-1})(v)).
	\]
\end{definition} 

\begin{remark}\label{coind_alg_remark}
	If $A_0$ is an algebra on which $\h_{\ol{0}}$ acts by derivations, then $\Coind_{\h_{\ol{0}}}^{\h}A_0$ is naturally a superalgebra with multiplication given by
	\[
	(fg)(u)=m_{A_0}\circ(f\otimes g)(\Delta(u)).
	\]
	In this case, the action of $\h$ on $\Coind_{\h_{\ol{0}}}^{\h}A_0$ is by super derivations.  A similar statement holds if $A_0$ is a Hopf algebra which $\h_{\ol{0}}$ acts on by derivations preserving the Hopf algebra structure.
\end{remark}

\begin{lemma}\label{coind_dual_ind_iso}
	For an $\h_{\ol{0}}$-module $V_0$, we have a canonical isomorphism of $\h$-modules
	\[
	(\Coind_{\h_{\ol{0}}}^{\h}V_0)^*\cong\Ind_{\h_{\ol{0}}}^{\h}(V_0)^*.
	\]
	This extends to an isomorphism of $H$-modules when $V_0$ is an $H_0$-module.
\end{lemma}

\begin{proof}
	We always have a canonical map of $\h$-modules (or $H$-modules when applicable)
	\[
	\Ind_{\h_{\ol{0}}}^{\h}(V_0)^*\to (\Coind_{\h_{\ol{0}}}^{\h}V_0)^*
	\]
	given by
	\[
	(u\otimes \varphi)(\eta)=(-1)^{\ol{u}(\ol{\varphi}+\ol{\eta})}\varphi(\eta(s(u))),
	\]
	where $s:\UU\h\to\UU\h$ is the antipode.  Here it is an isomorphism by the PBW theorem, since $\UU\h$ is a finitely generated free $\UU\h_{\ol{0}}$-module.
	
\end{proof}

\subsection{Koszul's induced superspace}
\begin{definition}
For an affine variety $X_0$ with an action of $\h_{\ol{0}}$, define $X_0^{(\h)}$ to be the affine variety with coordinate ring 
\[
k[X_0^{(\h)}]:=\Hom_{\h_{\overline{0}}}(\UU\h, k[X_0])=\Coind_{\h_{0}}^{\h}k[X_0].
\]
By Remark \ref{coind_alg_remark}, $X_0^{(\h)}$ is an affine supervariety with an action by $\h$.  If the action of $\h_{\ol{0}}$ on $X_0$ comes from an action of $H_0$ on $X_0$, then $X_0^{(\h)}$ will be an $H$-supervariety.  We have that $(X_0^{(\h)})_0=X_0$, and the natural projection $ k[X_0^{(\h)}]\to k[X_0]$ is given by $\eta\mapsto\eta(1)$.
\end{definition}

\begin{remark}
	The construction of the induced superspace given above can be done for any supervariety $X_0$ with an $\h_{\ol{0}}$-action as follows.  We define $X_0^{(\h)}$ to have underlying space $X_0$ and sheaf of functions given by, for an open subset $U_0\sub X_0$,
	\[
	\Gamma(U_0,\OO_{X_0^{(\h)}})=\Coind_{\h_{0}}^{\h}k[U_0]=k[(U_0)^{(\h)}].
	\]
	One can check this construction respects localization and thus gives a well-defined supervariety.
\end{remark}

\begin{remark}
	If $X_0$ is smooth, then $X_0^{(\h)}$ is also smooth.  In any case, under the action of $\h$ on $X_0^{(\h)}$, the elements of $\h_{\ol{1}}$ define everywhere non-vanishing vector fields on $X_0^{(\h)}$ so that $\h_{\ol{1}}\to (T_xX_0^{(\h)})_{\ol{1}}$ is an isomorphism for all $x\in X_0(k)$.
\end{remark}
\begin{remark}
	In \cite{koszul1982graded}, it was shown that $G=(G_0)^{\g}$.  Thus we have an explicit description of the algebra of functions on $k[G]$ given by
\[	
k[G]=\Hom_{\g_{\ol{0}}}(\UU\g,k[G_0]).
\]
From this one can also determine the Hopf algebra structure on $k[G]$ from the Hopf algebra structures on $\UU\g$ and $k[G_0]$.
\end{remark}

\begin{prop}\label{univ prop induced space} Let $X_0$ be an affine variety with an action of $\h_{\ol{0}}$.  Then $X_0^{(\h)}$ has the following universal property: given an affine supervariety $Y$ with an action of $\h$, and an $\h_{\ol{0}}$-equivariant map $\ol{\phi}:X_0\to Y_0$, there exists a unique $\h$-equivariant map $\phi:X_0^{(\h)}\to Y$ such that $\phi_0=\ol{\phi}$ and the following diagram commutes:
	\[
	\xymatrix{X_0^{(\h)} \ar[r]^{\phi} & Y\\
		X_0\ar[u]^{i_{X_0^{(\h)}}}\ar[r]^{\ol{\phi}} & Y_0\ar[u]_{i_{Y}}.}
	\]
	If the actions of $\h_{\ol{0}}$ and $\h$ integrate to actions of $H_0$ and $H$, respectively, then $\phi$ is a morphism of $H$-supervarieties.
\end{prop}
\begin{proof}
	Since everything is affine, we may work instead with the algebras of functions.  Then the commutativity of the square says that for such a map $\phi$, we must have that for $f\in k[Y]$,
	\[
	\phi^{*}(f)(1)=(\ol{\phi})^{*}\circ(i_{Y_0})^{*}(f)
	\]
	The property of begin an $\h$-equivariant map then forces, for $u\in\UU\h$,
	\[
	\phi^{*}(f)(u)=(-1)^{\ol{u}\ol{f}}(u\phi^{*}(f))(1)=(-1)^{\ol{u}\ol{f}}\phi^{*}(uf)(1)=(-1)^{\ol{u}\ol{f}}(\ol{\phi})^{*}((i_{Y_0})^{*}(uf))
	\]
	so the definition of $\phi^{*}$ is forced on us.  One can check that the above definition defines an algebra homomorphism, and so we are done.
\end{proof}

For the following, observe that if $X_0$ is an $H_0$-variety and $x$ is a closed point of $X_0$ which is fixed by $H_0$, then $H_0$ preserves $_0\m_x$ and thus acts on $\Dist(X_0,x)$.  
\begin{prop}\label{dist_induce}
	Let $X_0$ be affine $H_0$-variety, and $x\in X_0(k)$ a closed point which is fixed by $H_0$.  Then we have an isomorphism of $H$-modules 
	\[
	\Dist(X_0^{(\h)},x)\cong\Ind_{\h_{\ol{0}}}^{\h}\Dist(X_0,x).
	\]
	(See Remark \ref{Dist_G_module} for the $H$-module structure on $\Dist(X_0^{(\h)},x)$.)
\end{prop}
\begin{proof}
	We have the following string of isomorphisms of $H$-modules:
	\begin{eqnarray*}
		\Dist(X_0^{(\h)},x)& \cong &\Gamma_{\m_x}( k[X_0^{(\h)}])^*\\
		& = &\Gamma_{\m_x}(\Coind_{\h_{\ol{0}}}^{\h} k[X_0])^*\\
		& \cong &\Gamma_{\m_x}\Ind_{\h_{\ol{0}}}^{\h}( k[X_0])^*\\
		& = &\Ind_{\h_{\ol{0}}}^{\h}\Gamma_{_0\m_x}( k[X_0])^*\\
		& = &\Ind_{\h_{\ol{0}}}^{\h}\Dist(X_0,x)
	\end{eqnarray*}
	The non-trivial step here is that we can move $\Gamma$ past $\Ind$.  To see this, first observe that
	\[
	\m_x=\{\eta\in\Hom_{\h_{\ol{0}}}(\UU\h, k[X_0]):\eta(1)\in {_0\m_x}\}.
	\]
	It follows that for any $d\in\N$, for $n\gg d$, if $\eta\in\m_x^n$, then $\eta(u)\in$ $_0\m_x^d$ for all $u\in\UU\h$, because $\h_{\ol{0}}$ preserves the ideal $_0\m_x$.  Thus we have a natural inclusion
	\[
	\Ind_{\h_{\ol{0}}}^{\h}\Gamma_{_0\m_x}( k[X_0])^*\sub \Gamma_{\m_x}\Ind_{\h_{\ol{0}}}^{\h}( k[X_0])^*.
	\]
	To show it is surjective, recall that we have an identification $k[X_0^{(\h)}]\cong S\h_{\ol{1}}^*\otimes k[X_0]$ obtained by choosing a basis $u_1,\dots,u_n$ of $\h_{\ol{1}}$ and taking monomials $u_I=u_{i_1}\cdots u_{i_k}$, where $I=\{i_1<\dots<i_k\}$.  Let $\xi_{J,f}\in S\h_{\ol{1}}^*\otimes k[X_0]$ such that $\xi_{J,f}(u_I)=f$ for some function $f\in k[X_0]$ which we will choose later.  Then the monomials $s(u_I)$, where $s$ is the antipode, also projects to a basis of $S\h_{\ol{1}}$ in the associated graded.  Thus an arbitrary element in $\Ind_{\h_{\ol{0}}}^{\h}( k[X_0])^*$ may be presented as
	\[
	\phi=\sum\limits_I s(u_I)\otimes \varphi_I.
	\]
	Now suppose that $\phi(\m_x^N)=0$; let $f\in _0\m_x^N$; then $\xi_{J,f}\in\m_x^r$ for some $r>N$, so that $\phi(\xi_J)=0$.  On the other hand
	\[
	\phi(\xi_J)=(\sum\limits_I s(u_I)\otimes \varphi_I)(f\xi_J)=\pm \varphi_J(f)=0
	\]
	Since $f$ and $J$ were arbitrary, it follows that $\varphi_I( _0\m_x^r)=0$ for all $I$, as desired.
\end{proof}

\subsection{Induced vector bundles} Let us extend the above framework to the context of vector bundles.  Note that this subsection can largely be skipped, as it is only used to prove the $\Ind-\Coind$ isomorphism in Section \ref{subsec_ind_coind}, which is already known via other methods.  See section 4.6 of \cite{sherman2020sphericalthesis} for the definition of $\h$-equivariant vector bundle.

Again assume that $X_0$ is affine, and suppose that $F_0$ is an $\h_{\ol{0}}$-equivariant vector bundle over $X_0$.  Then $\h_{\ol{0}}$ acts on $\Gamma(X_0,F_0)$ and satisfies, for $f_0\in k[X_0]$, $s_0\in\Gamma(X_0,F_0)$, and $u_0\in\h_{\ol{0}}$,
\[
u_0\cdot(f_0s_0)=u_0(f_0)s_0+f_0u_0(s_0),
\]
where the action of $u_0$ on $f_0$ is coming from the action on $k[X_0]$.  If $F_0$ is actually an $H_0$-equivariant vector bundle, then for $h\in H_0$ we have
\[
h\cdot (f_0s_0)=L_h^*(f_0)h\cdot(s_0).
\] 
Then we define an $\h$-equivariant vector bundle $(F_0)^{(\h)}$ over $X_0^{(\h)}$ to have the space of sections given by
\[
\Gamma(X_0^{(\h)},(F_0)^{(\h)}):=\Coind_{\h_{\ol{0}}}^{\h}\Gamma(X_0,F_0)=\Hom_{\h_{\ol{0}}}(\UU\h,\Gamma(X_0,F_0)).
\]	
Then this admits a natural $\h$-action by virtue of being a coinduced module, and if $F_0$ is $H_0$-equivariant, $(F_0)^{(\h)}$ will be $H$-equivariant.  The $k[X_0^{(\h)}]$-module structure is defined by
\[
(\varphi s)(u)=a_{0,F_0}^*\circ (\varphi\otimes s)(\Delta(u))
\]
where $a_{0,F_0}^*:k[X_0]\otimes\Gamma(X_0,F_0)\to\Gamma(X_0,F_0)$ is the action map.  We check that this makes sense: for $v_0\in\h_{\ol{0}}$, writing $\Delta(u)=\sum u_i\otimes u^i$,
\begin{eqnarray*}
	(\varphi s)(v_0u)& = &a_0\circ (\varphi\otimes s)((v_0\otimes1+1\otimes v_0)\Delta(u))\\
	& = &(-1)^{\ol{s}\ol{u_i}}(\varphi(v_0u_i)s(u^i)+\varphi(u_i)s(v_0u^i))\\
	& = &(-1)^{\ol{s}\ol{u_i}}((v_0\varphi)(u_i)s(u^i)+\varphi(u_i)(v_0s)(u^i))\\
	& = &v_0((\varphi s)(u)).
\end{eqnarray*}
It's also straightforward to show that for $v\in\h$ we have
\[
v(\varphi s)=v(\varphi)s+(-1)^{\ol{v}\ol{\varphi}}\varphi v(s),
\]
and (when applicable) for $h\in H_0$ we have
\[
h\cdot (\varphi s)=L_h^*(\varphi)h\cdot s.
\]
This construction is local in the sense that we have, for an open subvariety $U_0\sub X_0$,
\[
\Gamma(U_0,(F_0)^{(\h)})=\Coind_{\h_{\ol{0}}}^{(\h)}\Gamma(U_0,F_0)=\Hom_{\h_{\ol{0}}}(\UU\h,\Gamma(U_0,F_0)).
\]
From this it is not hard to show that $(F_0)^{(\h)}$ is an even rank vector bundle on $X_0^{(\h)}$ of the same rank as $F_0$. 

\begin{prop}\label{bundles_on_induced_space}  There is a natural, rank preserving bijection between $\h$-equivariant vector bundles on $X_0^{(\h)}$ and $\h_{\ol{0}}$-equivariant vector bundles on $X_0$, given explicitly as follows:
	\begin{itemize}
		\item given an $\h_{\ol{0}}$-equivariant vector bundle $F_0$ on $X_0$, we may produce the $\h$-equivariant vector bundle $(F_0)^{(\h)}$ on $X$; and
		\item given an $\h$-equivariant vector bundle $F$ on $X_0^{(\h)}$ we may take the $\h_{\ol{0}}$-equivariant vector bundle on $X_0$ gotten by pulling back along embedding $i:X_0\hookrightarrow X_0^{(\h)}$.
	\end{itemize}
This restricts to a bijection of $H$-equivariant vector bundles on $X_0^{(\h)}$ and $H_0$-equivariant vector bundles on $X_0$, when applicable.
\end{prop}
\begin{proof}
	If we start with an $\h_{\ol{0}}$-equivariant vector bundle $F_0$, then we have a natural surjection 
	\[
	\Gamma(X_0,i^*(F_0)^{(\h)})\to\Gamma(X_0,F_0)
	\]
	given by evaluating an element of $\Hom_{\h_{\ol{0}}}(\UU\h,\Gamma(X_0,F_0))$ at $1\in\UU\h$. It is an $\h_{\ol{0}}$-equivariant map, and since the vector bundles have the same rank, it must be an isomorphism.
	
	In the other direction, suppose we have an $\h$-equivariant bundle $F$.  Then we have a natural map of $\h$-equivariant vector bundles $F\to(i^*F)^\h$, given on sections
	\[
	\Gamma(X,F)\to\Hom_{\h_{\ol{0}}}(\UU\h,\Gamma(X,i^*F))
	\] 
	by
	\[
	s\mapsto (u\mapsto i^*(us))	
	\]
	To see this is an isomorphism, it suffices to look at the map on fibers, which are isomorphisms.
\end{proof}

\begin{prop}\label{dist_bundle_induced}
	Under the same hypothesis of Proposition \ref{dist_induce}, let $F_0$ be an $H_0$-equivariant vector bundle on $X_0$.  Then we have a natural isomorphism of $H$-modules
	\[
	\Dist((F_0)^{(\h)},x)\cong\Ind_{\h_{\ol{0}}}^{(\h)}\Dist(F_0,x)
	\]
\end{prop}
\begin{proof}
	The proof extends almost verbatim from Proposition \ref{dist_induce}.
\end{proof}

\section{Induced spaces and homogeneous spaces}

\subsection{A local isomorphism for certain homogeneous supervarieties} Given a supergroup $G$ and a closed algebraic subgroup $K$, we may form the homogeneous supervariety $G/K$.  For technical aspects of such spaces see \cite{masuoka2018geometric} and \cite{masuoka2010quotient}.  In particular $G/K$ is always smooth and $(G/K)_0=G_0/K_0$, so that $G/K$ is affine if and only if $G_0/K_0$ is.  

Let $X=G/K$ be a homogeneous affine supervariety, and suppose that there exists a subgroup $H$ of $G$ such that $\h_{\ol{1}}\oplus\k_{\ol{1}}=\g_{\ol{1}}$.  Then consider the $H$-supervariety $(G_0/K_0)^{(\h)}$.  By its universal property (Proposition \ref{univ prop induced space}) it admits a canonical $H$-equivariant morphism to $G/K$.

\begin{prop}\label{local_iso}
	The canonical $H$-equivariant map $\phi:(G_0/K_0)^{(\h)}\to G/K$ induces an isomorphism of supervarieties in a Zariski open neighborhood of $eK_0$.  In particular, the map on functions
	\[
	\phi^{*}: k[G/K]\to k[(G_0/K_0)^{(\h)}]=\Hom_{\h_{\ol{0}}}(\UU\h, k[G_0/K_0])
	\]
	is an injective $H$-module homomorphism
\end{prop}

First we need a lemma.

\begin{lemma}\label{lemma_for_iso}
	Suppose that $f:X\to Y$ is a morphism of smooth affine supervarieties such that $f_0:X_0\to Y_0$ is an isomorphism and $df_x$ is an isomorphism for all closed points $x\in X(k)$.  Then $f$ is an isomorphism. 
\end{lemma}
\begin{proof}
Because $X$ and $Y$ are smooth and affine, we may present them as exterior algebras of vector bundles $E_{X_0}$, $E_{Y_0}$ on $X_0$, $Y_0$.  Working locally on an open cover, we may assume these vector bundles are trivial.  

Let $\xi_1,\dots,\xi_n\in \Gamma(Y_0,E_{Y_0})\sub k[Y]_{\ol{1}}$ be a $k[Y_0]$-basis for $\Gamma(Y_0,E_{Y_0})$ so these elements project to a basis of $(T_yY)_{\ol{1}}$ for each $y\in Y(k)$.  Then $f^*(\xi_1),\dots,f^*(\xi_n)\in k[X]_{\ol{1}}$ must project to a basis of $T_xX$ for all $x\in X(k)$.  It follows that $f^*(\xi_1),\dots,f^*(\xi_n)$ project to a $k[X_0]$-basis of $\Gamma(X_0,E_{X_0})$ in the associated graded of $k[X]$.  Hence the associated graded morphism of $f^*$ is an isomorphism, which implies that $f^*$ is an isomorphism, and we are done.
\end{proof}

\begin{proof}[Proof of Proposition \ref{local_iso}]
	Each space admits an action by $\h$ as vector fields acting on functions, and since $\phi^{*}$ is a $\h$-homomorphism we have 
	\[
	\phi^{*}(uf)=u\phi^{*}(f)
	\]
	for $u\in\h,f\in k[G/K]$.  Hence for any closed point $x$ of $G/K$ the following diagram is commutative:
	\[
	\xymatrix{\h\ar[d]\ar[dr]&\\
		T_x(G_0/K_0)^{(\h)} \ar[r]^{df_x} &T_x(G/K) }
	\]
	In particular, wherever $\h_{\ol{1}}$ spans the odd part of the tangent space of $G/K$ at $x$, $df_x$ will be an isomorphism of vector spaces when restricted to the odd part.  However, we see that $f_0^{*}$ is the identity map by construction, hence we get a commutative diagram:
	\[
	\xymatrix{(G_0/K_0)^{(\h)} \ar[r]^f & G/K\\
		G_0/K_0\ar[u]\ar[r]^{f_0=\id} & G_0/K_0.\ar[u]}
	\]
	If we restrict to the open affine neighborhood of $eK_0$ upon which $\h_{\ol{1}}$ spans the odd tangent space of $G/K$, $f$ will be an isomorphism by Lemma \ref{lemma_for_iso}.
\end{proof}

\begin{cor}\label{dist_iso}
Maintain the assumptions of Proposition \ref{local_iso}. and suppose further that $H_0\sub K_0$.  Then we have a canonical isomorphism of $H$-modules
	\[
	\Dist(G/K,eK)\cong\Ind_{\h_{\ol{0}}}^{\h}\Dist(G_0/K_0,eK_0).
	\]
	Or in terms of enveloping algebras,
	\[
	\UU\g/(\UU\g\k)\cong\Ind_{\h_{\ol{0}}}^{\h}\UU\g_{\ol{0}}/(\UU\g_{\ol{0}}\k_{\ol{0}}).
	\]
\end{cor}

\begin{proof}
	The first statement follows from Proposition \ref{dist_induce}.  The second statement follows from the usual identification of $\UU\g$-modules
	\[
	\Dist(G/K,eK)\cong \UU\g/(\UU\g\k).
	\]
	which comes from the composition
	\[
	\UU\g\to\Gamma(G/K,\DD_{G/K})\xto{\operatorname{res}_{eK}}\Dist(G/K,eK).
	\]
\end{proof}

\subsection{Supersymmetric spaces}  From now on, we assume that $G$ is a connected supergroup, i.e. an affine algebraic supergroup such that $G_0$ is connected.  Let $\theta$ be an involution of $G$, and let $K$ be a closed subgroup of $G$ such that $(G^\theta)^0\sub K\sub G^\theta$.  In particular $K$ need not be connected.  From this we may consider the homogeneous supervariety $G/K$, and we call $G$-supervarieties of this form symmetric supervarieties (or supersymmetric spaces).

On the level of Lie superalgebras, $\theta$ induces an involution on $\g$, which by abuse of notation we again write as $\theta$, giving rise to the decomposition $\g=\k\oplus\p$, where $\k=\operatorname{Lie}K$ is the fixed subspace of $\theta$ and $\p$ the $(-1)$-eigenspace of $\theta$. 
\begin{definition}
	Define $\k':=\k_{\ol{0}}\oplus\p_{\ol{1}}$, which is the fixed points of the involution $\delta\circ\theta$, where $\delta(x)=(-1)^{\ol{x}}x$ is the grading operator on $\g$.  Let $K'$ denote the closed algebraic subgroup of $G$ such that $K'_0=K_0$ and $\operatorname{Lie}K'=\k'$.  
\end{definition} 
Notice that $\delta\circ\theta$ is an involution on $G$ such that $(G^{\delta\circ\theta})^{0}\sub K'\sub G^{\delta\circ\theta}$, hence $G/K'$ is another symmetric supervariety.

\begin{prop}\label{symm_space_local_iso}
	We have a $K'$-equivariant morphism
	\[
 (G_0/K_0)^{(\k')}\to G/K
	\]
	which is an isomorphism in a neighborhood of $eK_0$.  In particular, the pullback morphism of functions 
	\[
	 k[G/K]\to k[(G_0/K_0)^{(\k')}]=\Coind_{\k_{\ol{0}}}^{\k'} k[G_0/K_0]
	\]
	is injective.  
\end{prop}
\begin{proof}
	This follows immediately from Proposition \ref{local_iso}.
\end{proof}
\begin{cor}\label{K'_invts}
	We have a natural injective morphism of algebras
	\[
	 k[K'\backslash G/K]= k[G/K]^{K'}\hookrightarrow  k[K_0\backslash G_0/K_0]= k[G_0/K_0]^{K_0}.
	\]
	In particular, $ k[K'\backslash G/K]$ is an integral domain.
\end{cor}
\begin{proof}
	Taking $K'$ invariants of the pullback morphism we obtain an injection
	\[
	 k[G/K]^{K'}\to \left(\Coind_{\k_{\ol{0}}}^{\k'} k[G_0/K_0]\right)^{K'}.
	\]
	Now one may use Frobenius reciprocity to identify $\left(\Coind_{\k_{\ol{0}}}^{\k'} k[G_0/K_0]\right)^{K'}$ with $ k[G_0/K_0]^{K_0}$ as an algebra.
\end{proof}

The following result, which now is proven easily from Corollary \ref{dist_iso}, is a generalization of the fundamental observation made in \cite{gorelik2000ghost}.
\begin{prop}\label{symm_space_dist}
We have a canonical isomorphism of $K'$-modules
\[
\Dist(G/K,eK)\cong\Ind_{\k_{\ol{0}}}^{\k'}\Dist(G_0/K_0,eK_0).
\]
\end{prop}

\section{The symmetric space $G/G_0$}

Consider the involution $\theta=\delta$, the canonical grading operator on $\g$ which is defined by $\delta(x)=-x$.  In this case $G^{\delta}=G_0$, and the local isomorphism of Proposition \ref{symm_space_local_iso} becomes a global isomorphism of $G$-supervarieties (both consisting of just one point):
\[
G/G_0\cong (G_0/G_0)^{\g}.
\]
\subsection{$\Ind$-$\Coind$ isomorphism}\label{subsec_ind_coind} The homogeneous space $G/G_0$ has one closed point, which we will call $e$.  Let $V$ be a $G$-equivariant vector bundle on $G/G_0$, and write $V$ again for its space of sections.  Since $\m_e$ is nilpotent, we have the identification $\Dist(V,e)=V^*$. 

A $G_0$-equivariant vector bundle on $G_0/G_0$ is the same data as a finite-dimensional $G_0$-representation $V_0$, and its sections are again $V_0$.  Then Proposition \ref{dist_bundle_induced} tells us that
\[
\Dist(V_0^{(\g)},e)=(\Coind_{\g_{\ol{0}}}^{\g}V_0)^*\cong\Ind_{\g_{\ol{0}}}^{\g}V_0^*,	
\]
which we already showed in Lemma \ref{coind_dual_ind_iso}.  

We will write $V_0^{(\g)}$ once again for the sections of $V_0^{(\g)}$ on $G/G_0$ since this space has only one point.  Since $D_{G/G_0}$ is generated by $\UU\g$ and $ k[G/G_0]$ as an algebra, $V_0^{(\g)}$ will be a $D_{G/G_0}$-module on $G/G_0$.  

Thus $\Dist(V_0^{(\g)},e)$ is a $D_{G/G_0}$-module, and, being finite-dimensional, must also be coherent.  By Proposition \ref{coherent_D_module_free}, $\Dist(V_0^{(\g)},e)$ must be a vector bundle on $G/G_0$, and is $G$-equivariant via its $\DD_{G/G_0}$-module structure.  Therefore by Proposition \ref{bundles_on_induced_space} there exists a $G_0$-equivariant vector bundle on $G_0/G_0$, that is a $G_0$-representation $W_0$, such that
\[
\Dist(V_0^{(\g)},e)\cong W_0^{(\g)}.
\]
By Proposition \ref{bundles_on_induced_space}, $W_0$ is obtained by taking the sections of the restriction of $V_0^{(\g)}$ to $G_0/G_0$, i.e.
\[
W_0\cong \Dist(V_0^{(\g)},e)/\m_e\Dist(V_0^{(\g)},e).
\]
Now let $n=\dim\g_{\ol{1}}$ so that $\m_e^{n+1}=0$ but $\m_e^n\neq 0$.  Then $\Dist(V,e)\cong(V_0^{(\g)}/\m_e^{n+1}V_0^{(\g)})^*$, and therefore
\[
(V_0^{(\g)}/\m_e^{n+1}V_0^{(\g)})^*/\m_e(V_0^{(\g)}/\m_e^{n+1}V_0^{(\g)})^*\cong (\m_e^nV_0^{(\g)})^*.
\]  
Now as a $G_0$-module,
\[
V_0^{(\g)}=\Hom_{\g_{\ol{0}}}(\UU\g,V_0)\cong \Hom(\Lambda\g_{\ol{1}},V_0)\cong \Lambda\g_{\ol{1}}^*\otimes V_0
\]
and $\m_e^{n}V$ sits inside as the $\g_{\ol{0}}$-submodule $\Pi^n\Lambda^n\g_{\ol{1}}^*\otimes V_0$.  Therefore, $W_0\cong\Pi^n\Lambda^n\g_{\ol{1}}\otimes V_0^*=\Ber(\g_{\ol{1}})\otimes V_0^*$.

Putting everything together and replacing $V_0$ by $V_0^*$, we have reproduced the following well-known result (see for instance section 9 of \cite{serganova2011quasireductive} for an algebraic proof):
\begin{prop}\label{ind-coind_iso}
	For a $G_0$-module $V_0$, we have a canonical isomorphism of $G$-modules 
	\[
	\Ind_{\g_{\ol{0}}}^{\g}V_0\cong\Coind_{\g_{\ol{0}}}^{\g}(\Ber(\g_{\ol{1}})\otimes V_0).
	\]
\end{prop}

\subsection{The invariant differential ghost operator} Now we study differential operators on $G/G_0$. By Proposition \ref{diff_ops_G/K}, the invariant differential operators are given by $G_0$-invariant distributions, i.e. $(\UU\g/(\UU\g)\g_{\ol{0}})^{G_0}$ which can be identified, via symmetrization, with $S(\g_{\ol{1}})^{G_0}$.  However, we also have $\Dist(G/G_0,e)\cong\Ind_{\g_{\ol{0}}}^{\g} k$ as a $\g$-module.  Therefore by Proposition \ref{ind-coind_iso} we have
\[
\Dist(G/G_0,e)\cong\Coind_{\g_{\ol{0}}}^\g(\Ber(\g_{\ol{1}}))
\]
Recall that $\ber_{\g}$ is the character of $G$ determined determined by $\Ber(\g)$.  If $V$ is a $G$-module, we write $V^{\ber_{\g}}$ for the subspace of eigenvectors for $G$ of weight $\ber_{\g}$. 

\begin{cor}\label{one_dim'l_semi_invts}
	Suppose that $\Lambda^{top}\g_{\ol{0}}$ is a trivial $G_0$-module.  Then the subspace $\Dist(G/G_0,e)^{\ber_{\g}}$ is one-dimensional.
\end{cor}
\begin{proof}
	Under this assumption, $\Ber(\g)=\Ber(\g_{\ol{1}})$, and so the result follows by Frobenius reciprocity.
\end{proof}
Observe that the conditions of Corollary \ref{one_dim'l_semi_invts} hold if $G_0$ is reductive.  When it exists, we write $v_{\g}\in\Dist(G/G_0,e)^{\ber_{\g}}$ for a chosen non-zero element.  If $\ber_{\g}=0$, i.e. $\Ber(\g_{\ol{1}})$ is a trivial $G_0$-module, then $v_{\g}\in\Dist(G/G_0,e)^{G}\sub\Dist(G/G_0,e)^{G_{0}}\cong D_{G/G_0}^G$. Therefore in this case $v_{\g}$ corresponds to a $G$-invariant differential operator on $G/G_0$, which we write as $D_{\g}$. 

\begin{definition}
	We say a supergroup $G$ is quasireductive if $G_0$ is reductive.  We say a Lie superalgebra $\g$ is quasireductive if it is the Lie superalgebra of a quasireductive supergroup.
\end{definition}

\textbf{Assumption}:  We assume for the rest of the section that $G$ is quasireductive and $\Ber(\g)=\Ber(\g_{\ol{1}})$ is the trivial $G_0$-module.  

Under this assumption, $v_{\g}$ corresponds, as we said, to a certain $G$-invariant differential operator $D_{\g}$.  We now determine what it is.

\begin{lemma}
	$D_{G/G_0}=\End_{ k}( k[G/G_0])$. 
\end{lemma}
\begin{proof}
	This easily follows from the definition of differential operators.
\end{proof}	

Recall that since $G_0$ is reductive, $ k[G/G_0]=\Coind_{\g_{\ol{0}}}^{\g} k$ is projective, and hence is a sum of injective indecomposable modules $I(L)$ for $L$ a simple $G$-module in the socle of $\Coind_{\g_{\ol{0}}}^{\g} k$.  We have that the trivial module, $ k$, shows up exactly once in $\Coind_{\g_{\ol{0}}}^{\g} k$, hence we can write:
\[
k[G/G_0]=\Coind_{\g_{\ol{0}}}^{\g} k=I(k)\oplus V.
\]
Since $\Ber(\g_{\ol{1}})$ is trivial, we have that $I(k)=P(k)$ (see \cite{serganova2011quasireductive}), and thus the head and tail of $I(k)$ are both the trivial module.   It follows that $\End(I(k))$ contains a unique up to scalar endomorphism $\phi$ taking the head to the tail, which is nilpotent exactly if $G$ does not have semisimple representations (i.e. $ k$ is not projective).

\begin{prop}
Up to a non-zero scalar, $D_{\g}$ is the endomorphism of $ k[G/G_0]=I(k)\oplus V$ given by $\phi\oplus 0_V$.
\end{prop}

\begin{proof}
	Since $D_{\g}$ is $G$-invariant by construction, it suffices to show that $\operatorname{res}_e(D_{\g})$ is $\g$-invariant.  However, we observe that $uD_{\g}=D_{\g}u=0$ for all $u\in\g$, so since $D_{G/G_0}\to\Dist(G/G_0,e)$ is right $D_{G/G_0}$-equivariant, we are done.
\end{proof}

The following is now easy to show:
\begin{cor}
For a $G_0$-module $V_0$, $D_{\g}$ acts on $\Coind_{\g_{\ol{0}}}^{\g} V_0$ by $\phi$ on each summand isomorphic to $I(k)$, and zero otherwise.
\end{cor}

We have the following characterization of $v_{\g}$:
\begin{cor}
	Let $\g$ be a Lie superalgebra such that $\Lambda^{top}\g_{\ol{0}}$ is the trivial $G_0$-module, and write $v_{\g}$ for a non-zero element of $(\UU\g/\UU\g\g_{\ol{0}})^{\ber_{\g}}$.  Then $v_{\g}$ is the unique non-zero element up to scalar with the property that 
	\[
	uv_{\g}-\ber_{\g}(u)v_{\g}\in\UU\g\g_{\ol{0}}
	\]
	for all $u\in\g$.  
\end{cor}

\subsection{Relation to Gorelik's element $v_{\emptyset}$}
\begin{prop}\label{invariants_ind_coind_iso}
	Let $\g$ be a Lie superalgebra such that $\Lambda^{top}\g_{\ol{0}}$ is the trivial $G_0$-module, so that $v_{\g}$ exists.  Then for a $G_0$-module $V_0$ we have a natural isomorphism
	\[
	(V_0)^{G_0}\to (\Ind_{\g_{\ol{0}}}^{\g}V_0)^{\ber_{\g}}
	\]
	given by
	\[
	z\mapsto v_{\g}z.
	\]
\end{prop}

\begin{proof}
	This easily follows from the work in already done in this section.
\end{proof}
In \cite{gorelik2000ghost}, it was proven that if $\Ber(\g_{\ol{1}})$ is trivial then there exists an element $v_{\emptyset}\in\UU\g$ with the property that for a $\g_{\ol{0}}$-module $V$, the map
\[
z\mapsto v_{\emptyset}z
\]
defines an isomorphism 
\[
 V^{\g_{\ol{0}}}\to (\Ind_{\g_{\ol{0}}}^{\g}V)^{\g}.
\]
\begin{cor}
The element $v_{\g}$ agrees with the construction of such an element given by Gorelik in \cite{gorelik2000ghost}.	
\end{cor}
\begin{proof}
	Gorelik's element has the property that it defined a nonzero element of $(\UU\g/\UU\g\g_{\ol{0}})^{\g}$, so we are done.
\end{proof}

\subsection{Computations of $v_{\g}$ for some Lie superalgebras}

We compute explicitly the element $v_{\g}\in\left(\UU\g/\UU\g\g_{\ol{0}}\right)^{\ber_{\g}}$ for certain quasireductive Lie superalgebras.  First, an easy example:
\begin{lemma}\label{v_for_almost_odd_abelian}
	If $[\g_{\ol{1}},\g_{\ol{1}}]$ is central in $\g$, then $v_{\g}=v_1\cdots v_n$ for any choice of basis $v_1,\dots,v_n$ of $\g_{\ol{1}}$.
\end{lemma}
\begin{proof}
	This element is acted on by $\g_{\ol{0}}$ according to its action on $\Ber(\g)$.  Observe that
	\[
	v_iv_{\g}=(-1)^{i-1}v_1\cdots v_i^2\cdots v_n+\sum\limits_{1\leq j<i}(-1)^{j-1}v_1\dots v_{j-1}[v_i,v_j]v_{j+1}\cdots v_n.
	\]
	Since both $[v_j,v_i]$ and $v_i^2$ are central, we may rewrite this sum as 
	\[
	(-1)^{i-1}v_1\cdots\widehat{v_i}\cdots v_nv_i^2+(-1)^{j-1}v_1\dots v_{j-1}\widehat{v_j}v_{j+1}\cdots v_n[v_j,v_i]\in\UU\g\g_{\ol{0}}.
	\]
\end{proof}

For type I algebras, we have the following:
\begin{prop}\label{v_g_for_type_I}
	Suppose that $\g$ is a quasireductive Lie superalgebra with a $\Z$-grading $\g=\g_{-1}\oplus\g_0\oplus\g_{1}$ such that $[\g_i,\g_j]\sub\g_{i+j}$, $\g_{\ol{0}}=\g_0$ and $\g_{\ol{1}}=\g_{-1}\oplus\g_{1}$.  Suppose further that for an odd weight $\alpha$ with respect to a Cartan subalgebra of $\g_{\ol{0}}$, $[\g_{\alpha},\g_{-\alpha}]$ acts trivially on $\Lambda^{top}\g_1$.
Let $v_1,\dots v_n$ be any basis of $\g_{-1}$ and $w_1,\dots,w_m$ any basis of $\g_1$.  Then
	\[
	v_\g=(v_1\cdots v_n)\cdot(w_1\cdots w_m).
	\]
\end{prop}
\begin{proof}
	Clearly $\g_{\ol{0}}$ acts on it as the top exterior power of $\g_{\ol{1}}$, and $\g_{-1}$ annihilates it.  Therefore it remains to show that $\g_1$ also annihilates it. Let $w\in\g_{1}$ be a weight vector, of weight $\alpha$ (note that $\alpha$ could be 0).  The above expression is seen to be independent of the choice of bases up to a nonzero scalar, so let us assume that $v_1,\dots,v_n$ are weight vectors and that $v_1,\dots,v_i$ are a basis of $\g_{-\alpha}$ (and if $\g_{-\alpha}=0$ then this condition is vacuous).   We see that (working up to $\UU\g\g_{\ol{0}}$)
	\begin{eqnarray*}
		w(v_1\cdots v_n)\cdot(w_1\cdots w_m)& = &\sum\limits_{j}(-1)^{j-1}v_1\cdots v_{j-1}[w,v_j]v_{j+1}\cdots v_nw_1\cdots w_m\\
		& = &\sum\limits_{j\leq i}(-1)^{j-1}v_1\cdots v_{j-1}[[w,v_j],v_{j+1}\cdots v_nw_1\cdots w_m]\\
		& + &\sum\limits_{j>i,k>j}(-1)^{j-1}v_1\cdots v_{j-1}v_{j+1}\cdots v_{k-1}[[w,v_j],v_k]v_{k+1}\cdots v_nw_1\cdots w_m\\
		& + &\sum\limits_{j>i}(-1)^{j-1}v_1\cdots v_{j-1}v_{j+1}\cdots v_n[[w,v_j],w_1\cdots w_m].
	\end{eqnarray*} 
In the first sum, we have that action of $[w,v_j]\in[\g_{\alpha},\g_{-\alpha}]$ on the vector $v_{j+1}\cdots v_nw_1\cdots w_m$ which has weight 
\[
j\alpha+\sum\limits_{\beta\in\Delta}\beta.
\]
By assumption, $\sum\limits_{\beta\in\Delta}\beta([\g_{\alpha},\g_{-\alpha}])=0$, and by the Jacobi identity we have $\alpha([\g_{\alpha},\g_{-\alpha}])=0$ as well.  Therefore this first term is zero.  

For the second sum, the weight of $[[w,v_j],v_k]$ is $\alpha+\alpha_j+\alpha_k$, where $\alpha_j$, resp. $\alpha_k$ is the weight of $v_j$, resp. $v_k$.  Since $\alpha\neq -\alpha_j,-\alpha_k$, we have that $\alpha+\alpha_j+\alpha_k\neq \alpha_j,\alpha_k$, and therefore $[[w,v_j],v_k]$ is either zero or a root vector in $\g_{-1}$ which we may assume already appears in the product $v_1\cdots v_{j-1}v_{j+1}\cdots v_{k-1}v_{k+1}\cdots v_n$, giving zero.

For the final sum, we know that $[w,v_j]$ is a nilpotent element of $\g_{\ol{0}}$ and thus acts trivially on $\Lambda^{top}\g_{1}$, so we once again get zero.
\end{proof}

\begin{cor}
	For $\g\l(m|n)$, $(\p)\s\l(m|n)$, $\o\s\p(2|2n)$, and $(\s)\p(n)$ the formula for $v_{\g}$ is given by Proposition \ref{v_g_for_type_I}.
\end{cor}
\begin{proof}
	It is straightforward to check the conditions of Proposition \ref{v_g_for_type_I} for these superalgebras.
\end{proof}

\begin{prop}\label{v_for_osp(1|2n)}
Let $\pm \delta_1,\dots,\pm\delta_n$ denote the odd roots of $\g=\o\s\p(1|2n)$ for the standard presentation of $\o\s\p(1|2n)$ as given for example in \cite{musson2012lie}.  Choose elements $u_1,\dots,u_n$ of weight $\delta_1,\dots,\delta_n$ and $v_1,\dots,v_n$ of weight $-\delta_1,\dots,-\delta_n$ such that if $h_i=[u_i,v_i]$, then $\delta_i(h_i)=1$.  Write $t_i=u_iv_i$.  Then we have
\[
v_{\g}=(1+t_1)(3+t_2)\cdots((2n-1)+t_n).
\]
\end{prop}
\begin{proof}
	This is in fact proven in section 4 of \cite{djokovic1976semisimplicity}.  There they prove that the above element is equal, mod $\UU\g\g_{\ol{0}}$, to
	\[
(1+t_{\sigma(1)})(3+t_{\sigma(2)})\cdots(2n-1+t_{\sigma(n)}).
	\]
	for any permutation $\sigma$. Now 
	\[
	u_i(1+t_i)=-v_iu_i^2, \ \ \ \ \ v_i(1+t_i)=-u_iv_i^2+v_ih_i.
	\]
	Since $u_i^2,v_i^2,$ and $h_i$ commute with $t_j$ for $j\neq i$, we may move them all the way to the right and obtain elements of $\UU\g\g_{\ol{0}}$.  
\end{proof}
It would be interesting to obtain formulas for $v_{\g}$ when $\g=\o\s\p(m|2n)$ with $m>2$, $\g=\q(n)$, or when $\g$ is exceptional simple.  Explicit formulas are important in computing the image of ghost distributions under the Harish-Chandra homomorphism, which we discuss later.  

\subsection{Semisimplicity criteria}

We give a brief application of the ideas above.

\begin{thm}\label{semisimplicity_criteria}
	Let $G$ be a quasireductive supergroup.  Then the following are equivalent:
	\begin{enumerate}
		\item The category $\operatorname{Rep}(G)$ of representations of $G$ is semisimple;
		\item $\Ber(\g_{\ol{1}})$ is trivial and $D_{\g}$ is not nilpotent.
		\item $\Ber(\g_{\ol{1}})$ is trivial and $\varepsilon(v_{\g})\neq0$, where $\varepsilon$ is the counit on $\UU\g$ (and is well-defined on $\UU\g/\g_{\ol{0}}\UU\g$).
	\end{enumerate}
\end{thm}
Note that the condition that $G$ be quasireductive is necessary in order for $\operatorname{Rep}(G)$ to be semisimple.
\begin{proof}
The equivalence $(3)\iff(2)$ is clear, so we show $(2)\iff(1)$.  Since $G$ is quasireductive, $ k[G/G_0]=\Coind_{\g_{\ol{0}}}^{\g} k$ is projective.  Since $\operatorname{Rep}(G)$ is semisimple if and only if the trivial module $ k$ is projective, it is equivalent to show that $ k$ splits off of $ k[G/G_0]$.   For this it is equivalent that $D_{\g}(1)\neq0$, which gives $(2)\iff(1)$.  
\end{proof}

It is well-known, going back to a result of Djokovic and Hochschild in \cite{djokovic1976semisimplicity}, that if $G$ is a connected algebraic supergroup such that $\operatorname{Rep}(G)$ is semisimple, then $G\cong K\times SOSp(1|2n_1)\times\dots\times SOSp(1|2n_k)$, where $K$ is a reductive Lie group.  Using Theorem \ref{semisimplicity_criteria} and Proposition \ref{v_for_osp(1|2n)} one can obtain a simple proof of this statement, and this has been carried out in \cite{sherman2020two}.

\section{General symmetric space $G/K$}


We come back to the general case of symmetric supervarieties $G/K$.  For the rest of the article we assume that $G$ is quasireductive, so that the connected component of the identity of $K$ is quasireductive, and in particular $K_0$ has semisimple representation theory. 

\subsection{Ghost distributions $\AA_{G/K}$}
We know by Proposition \ref{symm_space_dist} that
\[
\Dist(G/K,eK)\cong \Ind_{\k_0}^{\k'}\Dist(G_0/K_0,eK_0)
\]
as $K'$-modules.  Since $K_0$ has semisimple representation theory, $\Dist(G_0/K_0,eK_0)$ is a sum of finite-dimensional $K_0$-modules.  Hence by Proposition \ref{ind-coind_iso} we have that
\[
\Ind_{\k_0}^{\k'}\Dist(G_0/K_0,eK_0)\cong\Coind_{\k_0}^{\k'}(\Dist(G_0/K_0,eK_0)\otimes \Ber(\k'))
\]
where we have used that $\UU\k'$ is a finite rank, free left $\UU\k_0$-module.  Now we have
\begin{eqnarray*}
\Dist(G/K,ek)^{\ber_{\k'}}& \cong &(\Coind_{\k_0}^{\k'}(\Dist(G_0/K_0,eK_0)\otimes \Ber(\k'))^{\ber_{\k'}}\\
                      & \cong &(\Dist(G_0/K_0,eK_0)\otimes\Ber(\k'))^{\ber_{\k'}}\\
                      & \cong &\Dist(G_0/K_0,eK_0)^{K_0}
\end{eqnarray*}

Write $v_{\k'}$ for an element of $\UU\k'$ which projects to a nonzero element of $(\UU\k'/\UU\k'\k_{\ol{0}})^{\ber_{\k'}}$. Then by Proposition \ref{invariants_ind_coind_iso} we have that
\begin{prop}
	The isomorphism
	\[
	\eta:\Dist(G_0/K_0,eK_0)^{K_0}\to \Dist(G/K,eK)^{\ber_{\k'}}
	\]
	is given by
	\[
	z\mapsto v_{\k'}\cdot z.
	\]	
\end{prop}

\begin{remark}
	Recall we have an identification $\Dist(G_0/K_0,eK_0)^{K_0}\cong D^{G_0}(G_0/K_0)$, and this algebra is identified with $S(\a)^{W_{\a}}$, where $\a\sub\p_{\ol{0}}$ is a Cartan subspace and $W_{\a}$ is the little Weyl group associated to the symmetric space.
\end{remark}

Stated in terms of enveloping algebras, we have shown:

\begin{cor}
	We have an isomorphism
	\[
	(\UU\g_{\ol{0}}/\UU\g_{\ol{0}}\k_{\ol{0}})^{K_0}\to(\UU\g/\UU\g\k)^{\ber_{\k'}}
	\]
	given by $\phi(z)=v_{\k'}z$.
\end{cor}

\begin{definition}
	We define $\AA_{G/K}$ to be $\Dist(G/K,eK)^{\ber_{\k'}}$, and refer to elements of $\AA_{G/K}$ as ghost distributions on $G/K$.  We will often use the letter $\gamma$ to denote such a distribution.
\end{definition}
\begin{remark}[\textbf{Caution}]\label{A_notation_caution}
	In \cite{gorelik2000ghost}, $\AA$ is used to denote the $G'$ invariants in $\Dist(G\times G/G,eG)$ as we will see later on.  However in our notation, $\AA_{G\times G/G}$ denotes the $\ber_{\g'}$ semi-invariants of $G'$ acting on $\Dist(G\times G/G,eG)$.  Thus these will agree only when $\Ber(\g)$ is the trivial module.
	
  In Section \ref{ghost_centre_section}, we will introduce another object, $\AA_{\phi}$, which is a subspace of $\UU\g$ that is invariant under a certain twisted adjoint action depending on an automorphism $\phi$.  For this notation, we have that $\AA=\AA_{\delta}$.
\end{remark}

\subsection{Module structure of $\AA_{G/K}$}
Write $\ZZ_{G/K}$ for $\Dist(G/K,eK)^{K}=(\UU\g/\UU\g\k)^K$.  This is identified with the algebra of invariant differential operators on $G/K$, as explained in Proposition \ref{diff_ops_G/K}.  

\begin{prop}
	We have a natural map
	\[
	\AA_{G/K}\otimes\ZZ_{G/K}=(\UU\g/\UU\g\k)^{\ber_{\k'}}\otimes(\UU\g/\UU\g\k)^{K}\to (\UU\g/\UU\g\k)^{\ber_{\k'}}=\AA_{G/K}
	\]
	making $\AA_{G/K}$ into a right module over $\ZZ_{G/K}$.  
\end{prop}

\begin{proof}
	We define
	\[
	(\gamma+\UU\g\k)(z+\UU\g\k):=\gamma z+\UU\g\k
	\]
	Since $z$ is $K$-invariant, it is easy to check this is well-defined.  
\end{proof}
  \begin{lemma}\label{product_structure_general}
  	Suppose that $\Ber(\k_{\ol{1}})$ and $\Ber(\p_{\ol{1}})$ are trivial $K_0$-modules.  Then we have a natural map $\AA_{G/K}\otimes\AA_{G/K'}\to\ZZ_{G/K}$, or
  	\[
  	(\UU\g/\UU\g\k)^{K'}\otimes(\UU\g/\UU\g\k')^{K}\to(\UU\g/\k\UU\g)^{K},
  	\]
given by
  	\[
  	(\gamma+\UU\g\k)\otimes(\gamma'+\UU\g\k')\mapsto \gamma\gamma'+\k\UU\g.
  	\]
  \end{lemma}
\begin{proof}
	The proof is straightforward.
\end{proof}

\subsection{$ k[G/K]$ as a $K'$-module}\label{functions_as_K'_mod_section}

Observe that via pullback and the isomorphism of distributions we have the commutative diagram

\[
\xymatrix{\Dist(G/K,eK)^{\ber_{\k'}}\otimes k[G/K]\ar[r]\ar[dr] & \Dist((G_0/K_0)^{(\k')},eK_0)^{\ber_{\k'}}\otimes k[(G_0/K_0)^{(\k')}]\ar[d]\\
&  k.}
\]
Now let $\gamma$ be an element of $\Dist((G_0/K_0)^{(\k')},eK_0)^{\ber_{\k'}}$, and suppose that $\gamma(f)\neq0$ for $f\in k[(G_0/K_0)^{(\k')}]$.  Then necessarily the $K'$-module generated by $f$ generates a copy of $I_{K'}( k)$, the injective hull of $ k$ for $K'$.  It follows that the same must be true for the space $G/K$, and so we have:
\begin{prop}
	Let $\gamma\in\AA_{G/K}$, and suppose that $f\in k[G/K]$ is such that $\gamma(f)\neq0$.  Then the $K'$-module generated by $f$ contains a copy of $I_{K'}( k)$.  In particular, if $\k_{\ol{0}}\neq\g_{\ol{0}}$ then $ k[G/K]$ contains $I_{K'}( k)$ with infinite multiplicity.  
\end{prop}
\begin{proof}
	For the last statement, we observe that since $G/K$ is affine the pairing of distributions with functions is nondegenerate.  Using the fact that the $K'$-semi-invariant distributions form an infinite-dimensional vector space (being isomorphic, as a vector space, to $S(\a)^{W_{\a}}$), it is not hard to prove that the multiplicity must be infinite.
\end{proof}

\begin{remark}
	We observe that since $ k[G_0/K_0]^{K_0}$ is a subalgebra of $ k[G_0/K_0]$, 
	\[
	A:=\Coind_{\k_{\ol{0}}}^{\k'} k[G_0/K_0]^{K_0}
	\]
is a subalgebra of $ k[(G_0/K_0)^{(\k')}]$.  In particular $A$ is the sum of all copies of $I_{K'}( k)$ appearing in $ k[(G_0/K_0)^{(\k')}]$.  It follows that $ k[G/K]\cap A$ is a subalgebra of $ k[G/K]$ which contains all copies of $I_{K'}( k)$ in $ k[G/K]$, as well as all of $ k[G/K]^{K'}$.  The author has a rather limited understanding of $A\cap k[G/K]$.  
	
	Further we do not have a good answer, or even good formulation of the question of `how many' copies of $I_{K'}( k)$ are within $ k[G/K]$.  The copies of $I_{K'}( k)$ in $ k[(G_0/K_0)^{(\k')}]$ may be indexed by a basis of $ k[G_0/K_0]^{K_0}$, which itself is indexed by the irreducible summands of $ k[G_0/K_0]$, i.e. by certain dominant weights in $\a$.  For each copy of $I_{K'}( k)$ in $ k[G/K]$, one could record which dominant weights of $\a$ it is supported on in $ k[(G_0/K_0)^{(\k')}]$.  One could then look at all weights that appear in the supports of such copies of $I_{K'}( k)$ in $ k[G/K]$.  This collection of weights would be infinite and could not, for example, lie within any hyperplane.
	
	We also observe that since each copy of $I_{K'}( k)$ contains a $K'$-invariant function, we can deduce there are many $K'$-invariant functions on $ k[G/K]$ (recall we already know from Corollary \ref{K'_invts} that $ k[G/K]^{K'}$ is an integral domain).  Again the structure of $K[G/K]^{K'}$ is not generally understood, although it has been partially computed in certain examples, i.e. $G\times G/G$ and the superspheres $OSp(m|2n)/OSp(m-1|2n)$.  
\end{remark}

\begin{remark}
Suppose again the $\k_{\ol{0}}\neq \g_{\ol{0}}$.  Following a similar argument to the one above, one may deduce the existence of many projective $K'$-submodules of $ k[G/K]$ as follows: given an irreducible $K'$-submodule $L$ of $\Dist(G/K,eK)$, it defines an irreducible submodule of $\Dist((G_0/K_0)^{(\k')},eK_0)$ via our isomorphism.  Thus $L$ must pair nontrivially with some projective indecomposable summand $P=P_{K'}(V)$ of $ k[(G_0/K_0)^{(\k')}]$.  However this is only possible if $V\cong L^*$, and $L$ pairs with $P$ via $L\otimes P\to L\otimes L^*\to k$.  It follows that $ k[G/K]$ must contain a copy of $P_{K'}(L^*)$ as well.  
\end{remark}

\section{Pairs that have an Iwasawa decomposition}\label{section_iwasawa}

In this section we introduce Cartan subspaces, the Iwasawa decomposition, and the Harish-Chandra homomorphism.  

\subsection{Cartan subspaces and the Iwasawa decomposition} Continue to let $G$ be an arbitrary quasireductive supergroup with involution $\theta$ and a corresponding symmetric subgroup $K$.  As always we write $\g=\k\oplus\p$ for the eigenspace decomposition of $\theta$ on $\g$.  Choose a Cartan subspace $\a_{\ol{0}}\sub\p_{\ol{0}}$, and extend it to a $\theta$-stable Cartan subalgebra $\h_{\ol{0}}\sub\g_{\ol{0}}$.  Then if we let $\h=\c(\h_{\ol{0}})$, $\h$ will be a $\theta$-stable Cartan subalgebra of $\g$.  We may then write $\h=\t\oplus\a$ for the eigenspace decomposition of $\theta$, where $\t$ is fixed and $\a$ is the $(-1)$-eigenspace.  

\begin{definition}\label{defn_cartan_subspace}
	We define $\a$ to be a Cartan subspace of $\p$ for the pair $(\g,\k)$.
\end{definition}

\begin{remark}
	It is known that all choices of $\h_{\ol{0}}$ constructed in this way are conjugate under $K_0$, thus all Cartan subalgebras constructed in this way are too, so that a Cartan subspace $\a$ is well-defined up to conjugation by $K_0$.
\end{remark}
\begin{remark}[\textbf{Caution}] If $\a\neq\a_{\ol{0}}$ then $\a$ need not be abelian or even a subalgebra of $\g$: indeed, $[\a_{\ol{1}},\a_{\ol{1}}]\sub\t_{\ol{0}}$.
\end{remark}

We may decompose $\g$ into eigenspaces under the action of $\a_{\ol{0}}$, and write $\ol{\Delta}$ for the non-zero weights of this action, and call them restricted roots.  Note that the weights of this action are exactly the restriction to $\a_{\ol{0}}$ of the roots under the action of a maximal torus in $\g_{\ol{0}}$ which contains $\a_{\ol{0}}$. We may choose a decomposition $\ol{\Delta}=\ol{\Delta}^+\sqcup\ol{\Delta}^-$ into positive and negative roots.  Define
\[
\n=\bigoplus\limits_{\ol{\alpha}\in\ol{\Delta}^+}\g_{\ol{\alpha}}.
\]
\begin{definition}
	We say that the supersymmetric pair $(\g,\k)$ admits an Iwasawa decomposition if we have a decomposition $\g=\k\oplus\a\oplus\n$ for some choice of $\n$ as above.
\end{definition}

If $\b$ is the Borel subalgebra determined by this choice of positive roots then we have $\a\oplus\n\sub\b$.  In particular, $\b+\k=\g$, i.e. $\k$ has a complimentary Borel subalgebra.  
\begin{definition}
	We call a Borel subalgebra arising in this way an Iwasawa Borel subalgebra for $\k$ with respect to $\a$.  If $B\sub G$ integrates $\b$, then we call $B$ an Iwasawa Borel subgroup of $G$.
\end{definition}

\subsection{Irreducible $B$-submodules}  We suppose from now on that $(\g,\k)$ admits an Iwasawa decomposition. Since $(\g,\k)$ admits an Iwasawa decomposition, $G/K$ will admit an open orbit at $eK$ under an Iwasawa Borel subgroup $B$, whose Lie superalgebra contains $\a\oplus\n$.  We write $\Lambda^+\sub\h_{\ol{0}}^*$ for the set of $B$-dominant weights $\lambda$ such that there exists an irreducible $B$-submodule of highest weight $\lambda$ in $k[G/K]$.

Let $\lambda\in\Lambda^+$, and let $L_{\lambda}$ be an irreducible $B$-submodule of $k[G/K]$ of highest weight $\lambda$.  In particular $L_{\lambda}$ is an irreducible $\h$-module.  Because $B$ admits an open orbit at $eK$ we must have that $\operatorname{ev}_{eK}:L_{\lambda}\to k$ is non-zero.  It follows that the image of $L_{\lambda}$ in the projection to $k[G_0/K_0]$ must be non-zero, and its image will contain a $B_0$-highest weight function of weight $\lambda$.  By what is known about symmetric spaces, this implies that $\lambda\in\a_{\ol{0}}^*,$ so in particular $\Lambda^+\sub\a_{\ol{0}}^*$, and in particular $\t_{\ol{0}}$ acts trivially on $L_{\lambda}$.

\subsubsection{Structure of $L_{\lambda}$} Consider the irreducible $\h$-module $L_{\lambda}$.  We have already noted that $\t_{\ol{0}}$ acts trivially, thus it is a module over $\h/\t_{\ol{0}}$.  Both $\a_{\ol{1}}$ and $\t_{\ol{1}}$ sit inside the quotient superalgebra as odd abelian subalgebras.  The action of $\h/\t_{\ol{0}}$ on $L_{\lambda}$ is given, up to parity, by the unique irreducible super representation of the Clifford superalgebra $\operatorname{Cl}(\h_{\ol{1}},(-,-)_{\lambda})$, where 
\[
(u,v)_{\lambda}=\lambda([u,v]).
\]  
Thus $W_{\lambda}:=\ker(-,-)_{\lambda}$ acts trivially on $L_{\lambda}$, and $\a_{\ol{1}}/(W_{\lambda}\cap\a_{\ol{1}})$ and $\t_{\ol{1}}/(W_{\lambda}\cap\t_{\ol{1}})$ define complimentary maximal isotropic subspaces of $\h_{\ol{1}}/\ker(-,-)_{\lambda}$.

\begin{lemma}\label{structure_of_L_lambda} \
	The irreducible $\h$-module $L_{\lambda}$ is not isomorphic to its parity shift; further
	\begin{enumerate}
		\item as an $\a_{\ol{1}}$-module, $L_{\lambda}$ is isomorphic to $\Lambda^\bullet\left(\a_{\ol{1}}/(W_{\lambda}\cap\a_{\ol{1}})\right)$;
		\item as an $\t_{\ol{1}}$-module, $L_{\lambda}$ is isomorphic to $\Lambda^\bullet\left(\t_{\ol{1}}/(W_{\lambda}\cap\t_{\ol{1}})\right)$.
	\end{enumerate}
	The socle of $L_{\lambda}$ as an $\a_{\ol{1}}$-module generates $L_{\lambda}$ as a $\t_{\ol{1}}$-module, and the socle as a $\t_{\ol{1}}$-module generates it as an $\a_{\ol{1}}$-module.
\end{lemma}
\begin{proof}
	Both of these lemma follow from the fact that we may realize the irreducible representation of an even-dimensional Clifford algebra as an exterior algebra $\Lambda^\bullet\langle\xi_1,\dots,\xi_n\rangle$.  Then if $V_1,V_2$ are complimentary maximal isotropic subspaces, we may realize $V_1$ as acting by multiplication by $\xi_1,\dots,\xi_n$, and $V_2$ as acting by $\d_{\xi_1},\dots,\d_{\xi_n}$. 
\end{proof}

\subsection{Highest weight submodules of $ k[G/K]$} Observe that $\r=\t_{\ol{0}}\oplus\a\oplus\n$ is a subalgebra of $\g$; write $R$ for the subgroup of $G$ which integrates $\r$.  Since $\r_{\ol{1}}=(\a\oplus\n)_{\ol{1}}$ is complimentary to $\k_{\ol{1}}$ by the Iwasawa decomposition, we may apply Proposition \ref{local_iso} to obtain the existence of a canonical map of $R$-supervarieties
\[
(G_0/K_0)^{(\r)}\to G/K
\]
which is an isomorphism in a neighborhood of $eK_0$.  In particular the map on functions
\[
k[G/K]\to  k[(G_0/K_0)^{(\r)}]
\]
is injective map of $R$-modules. Now consider the subalgebra $\tilde{\r}:=\t_{\ol{0}}\oplus\a_{\ol{1}}\oplus\n$ of $\r$.  Then $\a_{\ol{0}}$ normalizes it, and thus we obtain an injective map of $\a_{\ol{0}}$-modules
\[
k[G/K]^{\tilde{\r}}\to  (k[(G_0/K_0)^{(\r)}])^{\tilde{\r}}.
\]
Let $\lambda\in\Lambda\sub\a_{\ol{0}}^*$, where $\Lambda$ is the $\Z$-span of the set of $B_0$-highest weights of $\C[G_0/K_0]$.   Let $k_{\lambda}$ be the one-dimensional purely even $R$-module of weight $\lambda$.  Then 
\[
\Hom_{\r}(k_{\lambda},k[(G_0/K_0)^{(\r)}])=\Hom_{\h_{\ol{0}}\oplus\n_{\ol{0}}}(k_{\lambda},k[G_0/K_0]).
\] 
Now the RHS is at most one-dimensional since $G_0/K_0$ is spherical. Thus the socle of $k[G/K]$ as a $B$-module, and thus also as a $G$-module, is multiplicity-free.

\begin{cor}\label{symmetric socle}
	If $(\g,\k)$ admits an Iwasawa decomposition, then:
	\begin{enumerate}
		\item $\operatorname{soc}\C[G/K]$ is multiplicity-free;
		\item if $L$ is a simple $G$-submodule of $\C[G/K]$, then $L\not\cong\Pi L$, and $\Pi L$ is not a submodule of $\C[G/K]$.
	\end{enumerate}
\end{cor}
\begin{remark}
	Corollary \ref{symmetric socle} is remarkable in that a $G$-supervariety with an open orbit under a Borel subgroup need not satisfy the above conditions if $\h\neq\h_{\ol{0}}$.  Such properties however are natural generalizations of properties enjoyed by spherical varieties in the classical setting.
\end{remark}

The above application of Frobenius reciprocity implies that in a neighborhood of $eK_0$ in $(G_0/K_0)^{(\r)}$, we have functions $f_{\lambda}$ annihilated by $\t_{\ol{0}}\oplus\a_{\ol{1}}\oplus\n$ and of weight $\lambda\in\a^*$ for every $\lambda\in\Lambda$.  Using the local isomorphism, we find that such functions also exist in a neighborhood of $eK$ in $G/K$; thus if we consider the $\h$-module they generate, we find that for all $\lambda\in\Lambda$ there exists an $(\h+\r)$-submodule $L_{\lambda}$ defined in some neighborhood of $eK$ which is irreducible over $\h$.	By the ideas used in \cite{sherman2019sphericalsupervar}, the modules $L_{\lambda}$ will be irreducible over $B$.

\subsection{The $K'$ submodule generated by $L_{\lambda}$} Recall that for a subspace $W$ we write $\langle K'\cdot W\rangle$ for the $K'$-module generated by $W$.
\begin{lemma}\label{k'_module_gend_same}
	Let $L_{\lambda}\sub k[G/K]$ be an irreducible $B$-submodule of $k[G/K]$. Then $\UU\k' L_{\lambda}$ is stable under $K'$, and thus 
	\[
	\langle K'\cdot L_{\lambda}	\rangle=\UU\k' L_{\lambda}.
	\]
\end{lemma}
\begin{proof}
	It suffices to prove that $\UU\k' L_{\lambda}$ is stable under $K_0$.  For this we notice that by the Iwasawa decomposition for $\g_{\ol{0}}$ and the fact that $G_0$ is connected,
	\[
	\langle G_0\cdot L_{\lambda}\rangle=\UU\g_{\ol{0}} L_{\lambda}=\UU\k_{\ol{0}} L_{\lambda}\sub \langle K_0\cdot L_{\lambda}\rangle.
	\]
	Since $\langle K_0\cdot L_{\lambda}\rangle\sub\langle G_0\cdot L_{\lambda}\rangle$ it follows that $\langle K_0\cdot L_{\lambda}\rangle=\UU\k_{\ol{0}}L_{\lambda}$.  From here it is easy to check that $\UU\k'L_{\lambda}$ is $K_0$-stable.
\end{proof}

\begin{remark}\label{k'_module_gend_is_g_module}  If $(\g,\k')$ also satisfies the Iwasawa decomposition then we have $\langle K'\cdot L_{\lambda}\rangle=\langle G\cdot L_{\lambda}\rangle$.  Also, it will be shown in the subsequent article that if $\g$ is basic classical and the involution $\theta$ preserves the nondegenerate form on $\g$, then we always have $\langle K'\cdot L_{\lambda}\rangle=\langle G\cdot L_{\lambda}\rangle$.  
\end{remark}

\begin{prop}
	$\UU\k'L_{\lambda}$ contains at most one copy of $I_{K'}(k)$.
\end{prop}
\begin{proof}
	The $\k'$-semicoinvariants of $\UU\k'L_{\lambda}$ are determined by their values on $L_{\lambda}$, and give rise to an $\a_{\ol{1}}$-coinvariant on $L_{\lambda}$.  However as an $\a_{\ol{1}}$-module $L_{\lambda}$ admits a unique $\a_{\ol{1}}$-coinvariant by Lemma \ref{structure_of_L_lambda}, so there can be at most one $\k'$-semicoinvariant on $\UU\k'L_{\lambda}$ up to scalar.  Thus at most one copy of $I_{K'}(k)$ can arise. 
\end{proof}

\begin{cor}
	Let $f\in L_{\lambda}^{\a_{\ol{1}}}$ be nonzero.  Then $f(eK)\neq0$.
\end{cor}

\begin{lemma}\label{appendix_ker_lemma}
	If $\UU\k'L_{\lambda}$ contains a copy of $I_{K'}(k)$, then $\ker(-,-)_{\lambda}\cap\a_{\ol{1}}=0$, or equivalently $L_{\lambda}\cong\Lambda^\bullet\a_{\ol{1}}$ as an $\a_{\ol{1}}$-module.
\end{lemma}

\begin{proof}
	Write $\UU\k'L_{\lambda}=I_{K'}(k)\oplus M$ for some complimentary submodule $M$.  Then the restriction of $I_{K'}(k)$ to $\t_{\ol{0}}\oplus\a_{\ol{1}}$ must remain injective.  In particular if $g\in I_{K'}(k)$ generates the head, then we will have $\Lambda^{top}\a_{\ol{1}}g\neq0$.  If we choose $f\notin\a_{\ol{1}}L_{\lambda}$, then $\UU\k'f$ will generate the copy of $I_{K'}(k)$ and thus we must also have $\Lambda^{top}\a_{\ol{1}}f\neq0$.  The statement follows.
\end{proof}

\subsection{Harish-Chandra homomorphism}\label{sec HC} Consider the Lie superalgebra $\t_{\ol{0}}\oplus\a$.  Then $\t_{\ol{0}}$ is a central subalgebra and contains the derived subalgebra, so we may take the quotient by it to obtain the abelian Lie superalgebra $(\t_{\ol{0}}\oplus\a)/\t_{\ol{0}}$; we take pains not to write this as $\a$, since as stated earlier $\a$ itself is not a Lie superalgebra.  We write $\mathfrak{A}:=\UU(\t_{\ol{0}}\oplus\a/\t_{\ol{0}})$ and we view it as the supersymmetric polynomial algebra on the underlying super vector space of $\a$.

If we restrict the natural map $\UU\g\to\UU\g/\UU\g\k$ to $\UU(\t_{\ol{0}}\oplus\a)$, we obtain the projection 
\[
\UU(\t_{\ol{0}}\oplus\a)\to\UU(\t_{\ol{0}}\oplus\a)/(\t_{\ol{0}})\cong\mathfrak{A},
\]
so that $\mathfrak{A}$ is naturally subspace of $\UU\g/\UU\g\k$.  Further, by the PBW theorem we have a decomposition
\[
\UU\g/\UU\g\k=\mathfrak{A}\oplus\n\UU\g/(\n\UU\g\cap\UU\g\k).
\]

\begin{definition}
	We define the Harish-Chandra homomorphism 
	\[
	HC:\Dist(G/K,eK)=\UU\g/\UU\g\k\to\mathfrak{A}
	\]
	to be the projection along $\n\UU\g/(\n\UU\g\cap\UU\g\k)$.  
\end{definition}

Now $\t_{\ol{0}}\oplus\a_{\ol{1}}$ is a subalgebra of $\k'$, and this acts on $\mathfrak{A}$ by left multiplication on the quotient, and $\t_{\ol{0}}$ acts trivially while $\a_{\ol{1}}$ acts freely.  Thus we obtain a free action of the exterior algebra on $\a_{\ol{1}}$ on $\mathfrak{A}$, and therefore the invariants of this action are given by
\[
\mathfrak{A}^{\a_{\ol{1}}}=S(\a_{\ol{0}})\Lambda^{top}\a_{\ol{1}}.
\]
Since the decomposition 
\[
\UU\g/\UU\g\k=\mathfrak{A}\oplus\n\UU\g/(\n\UU\g\cap\UU\g\k)
\]
is $\t_{\ol{0}}\oplus\a_{\ol{1}}$-invariant, we clearly obtain the following lemma.
\begin{lemma}\label{HC_ghost_to_invts}
	The Harish-Chandra morphism restricts to a map
	\[
	HC:\AA_{G/K}\to\mathfrak{A}^{\a_{\ol{1}}}=S(\a_{\ol{0}})\Lambda^{top}\a_{\ol{1}}.
	\]
\end{lemma}

\begin{remark}\label{poly_from_HC}
	By Lemma \ref{HC_ghost_to_invts}, it follows that given a ghost distribution $\gamma\in\AA_{G/K}$, we may obtain a polynomial in $S(\a_{\ol{0}})$ by writing $HC(\gamma)=p_{\gamma}\xi$, where $\xi\in\Lambda^{top}\a_{\ol{1}}$ is some chosen nonzero element.  Thus $p_{\gamma}$ is well-defined up to a scalar, so that its vanishing set is well-defined.    Although we caution that it depends on the choice of positive restricted roots.
\end{remark}

We have a linear map $\a_{\ol{1}}\otimes\a_{\ol{0}}^*\to\t_{\ol{1}}^*$ given by 
\[
(u\otimes \lambda)\mapsto (u,-)_{\lambda}.
\]
Let $U_{reg}\sub\a_{\ol{0}}^*$ denote the locus where this defines an injective morphism $\a_{\ol{1}}\to\t_{\ol{1}}^*$.  Equivalently, $\lambda\in U_{reg}$ if and only if the irreducible $\h$-module of weight $\lambda$ is projective over $\a_{\ol{1}}$.  Clearly $U_{reg}$ is Zariski open, although it need not be nonempty: if $\dim\a_{\ol{1}}>\dim\t_{\ol{1}}$, or if there exists a nonzero element $u\in\a_{\ol{1}}$ such that $[u,\t_{\ol{1}}]=0$, then $U_{reg}=\emptyset$.

\begin{remark}\label{rem dims equal}
	If $\dim\a_{\ol{1}}=\dim\t_{\ol{1}}$, then $U_{reg}$ is the complement of the vanishing set of a homogeneous polynomial of degree $\deg\a_{\ol{1}}$, which is given by the composition of $\a_{\ol{0}}^*\to\Hom(\a_{\ol{1}},\t_{\ol{1}}^*)$ with a determinant map (after choosing coordinates).
\end{remark}

\subsection{Degree bound} Choose a homogeneous basis $p_1,\dots,p_r$ of $\p$.  Recall from Lemma \ref{dist_symm_lemma} that $\Dist(G/K,eK)$ is spanned by monomials
\[
p_1^{k_1}\cdots p_r^{k_r},
\]
where we abusively identify them with their restrictions to $\Dist(G/K,eK)$.  Define a filtration $F^\bullet$ on $\Dist(G/K,eK)$ by setting $F^0=\langle\operatorname{ev}_{eK}\rangle$, and set $F^i$ to be the span of all monomials as above where the degree of $p_i$ is 2 if $p_i$ is even and 1 if $p_i$ is odd.

On the other hand, define a grading of $\mathfrak{A}$ by giving elements of $\a_{\ol{0}}$ degree 1, and elements of $\a_{\ol{1}}$ degree $1/2$. The following is a generalization of 4.2.2 in \cite{gorelik2000ghost}.

\begin{lemma}\label{degree_bound}
	We have
	\[
	HC(F^r)\sub\sum\limits_{s\leq r/2}\mathfrak{A}^{s}.
	\]
\end{lemma}

\begin{proof}
	Since the Harish-Chandra projection is linear, it suffices to prove this on monomials.  We induct on the length of the monomial.  If the monomial is length zero, the result is clear.  Suppose we have a monomial $b_1\cdots b_t\in F^{r}$.  Using the Iwasawa decomposition, we may write $b_1=k+a+n$ where $k\in\k$, $a\in\a$, and $n\in\n$.  Since $nb_2\cdots b_{t}\in\n\UU\g$ it vanishes under the Harish-Chandra projection, so we have
	\begin{eqnarray*}
		HC(b_1\dots b_t)& = &HC(ab_2\cdots b_t)+HC(kb_2\cdots b_t)\\
		& = &HC(ab_2\cdots b_{t})+\sum\limits_{i} HC(b_2\cdots[k,b_i]\cdots b_{t})+HC(b_2\cdots b_tk).
	\end{eqnarray*}
	Since $b_2\cdots b_tk\in\UU\g\k$, the last term vanishes.  If $b_t$ is even then $b_2\cdots b_{t}\in F^{r-2}$, and $HC(ab_2\cdots b_{t})=HC(a)HC(b_2\cdots b_{t})$. Thus by induction,
	\[
	\deg HC(ab_2\cdots b_t)\leq \deg HC(b_2\cdots b_{t})+1\leq \frac{r-2}{2}+1=\frac{r}{2}.
	\]
	On the other hand if $b_t$ is odd then $b_2\cdots b_{t}\in F^{r-1}$ and we have
	\[
	\deg HC(ab_2\cdots b_t)\leq \deg HC(b_2\cdots b_{t})+\frac{1}{2}\leq \frac{r-1}{2}+1=\frac{r}{2}.
	\]
	This deals with the first term.  As for the terms $HC(b_2\cdots[k,b_i]\cdots b_{t})$, if $b_1$ is even then the parity of $[k,b_i]$ is the same as the parity of $b_i$, and thus $b_2\cdots[k,b_i]\cdots b_{t}\in F^{r-2}$.  Therefore
	\[
	\deg HC(b_2\cdots[k,b_i]\cdots b_{t})\leq (r-2)/2\leq r/2.
	\]
	If $b_1$ is odd, then the parity of $[k,b_i]$ is opposite the parity of $b_i$, so $b_2\cdots[k,b_i]\cdots b_{t}\in F^{r}$, and by induction we have
	\[
	HC(b_2\cdots[k,b_i]\cdots b_{t})\leq r/2.
	\]
\end{proof}

\begin{cor}
	Let $z\in\Dist^r(G_0/K_0,eK_0)^{K_0}$ lie in the $r$th part of the standard filtration on $\Dist(G_0/K_0,eK_0)$ defined in  Definition \ref{dist_def}.  Then $v_{\k'}\cdot z\in\AA_{G/K}$ has 
	\[
	HC(v_{\k'}\cdot z)\in \mathfrak{A}^{r+\dim\p_{\ol{1}}/2}.
	\]
	In particular, in the notation of Remark \ref{poly_from_HC}, 
	\[
	\deg p_{v_{\k'}\cdot z}\leq \dim\n_{\ol{1}}/2+r.
	\]
\end{cor}
\begin{proof}
	The first statement follows from Lemma \ref{degree_bound}.  The second statement follows from Remark \ref{poly_from_HC}, where we showed that $HC(v_{\k'}\cdot z)=p_{\gamma}\xi$ for some $p_{\gamma}\in S(\a_{\ol{0}})$ and a non-zero element $\xi\in\Lambda^{top}\a_{\ol{1}}$.  Since $\deg\xi=\dim\a_{\ol{1}}/2$, and $\dim\p_{\ol{1}}=\dim\a_{\ol{1}}+\dim\n_{\ol{1}}$,  the bound follows.  
\end{proof}

\begin{lemma}\label{hc_image_module_structure}
	$HC(\AA_{G/K})$ is naturally a module over $HC(\ZZ_{G/K})$ such that $HC:\AA_{G/K}\to S(\a)$ induces a morphism of $\ZZ_{G/K}$-modules.
\end{lemma}
\begin{proof}
	Let $\gamma\in\AA_{G/K}$ and $z\in\ZZ_{G/K}$.  Write $z=n+HC(z)+\UU\g\k$ and $\gamma=n'+HC(\gamma)+\UU\g\k$.  Then we see that
	\[
	\gamma z=(n'+HC(\gamma)+\UU\g\k)(n+HC(z)+\UU\g\k)=n'(n+HC(z))+HC(\gamma)n+HC(z)HC(\gamma)+\UU\g\k.
	\]
	Clearly $n'(n+HC(z))\in\n\UU\g$.  And since $HC(\gamma)\in S(\a)$, it preserves $\n\UU\g$ under commutator and thus $HC(\gamma)n\in\n\UU\g$ as well.  Hence we find that
	\[
	HC(\gamma z)=HC(\gamma)HC(z),
	\]
	as desired.
\end{proof}

\subsection{Branching}\label{appendix_ghost_dist_subsec} 

Let $\gamma\in\AA_{G/K}$ and $L_{\lambda}$ an irreducible $B$-submodule of $k[G/K]$.  Then $HC(\gamma)$ defines a functional on $L_{\lambda}$.   
\begin{lemma}
	If $HC(\gamma):L_{\lambda}\to k$ is nonzero then $\UU\k'L_{\lambda}$ contains a copy of $I_{K'}(k)$. Further, $\UU\k'L_{\lambda}$ contains a copy of $I_{K'}(k)$ if and only if $HC(\gamma):L_{\lambda}\to k$ is non-zero for some $\gamma\in\AA_{G/K}$. 
\end{lemma}
\begin{proof}
	Follows from the work done in Section \ref{functions_as_K'_mod_section}.
\end{proof}

The following corollary follows from the work done so far in this section.
\begin{cor}\label{p_gamma_nonzero_condition}
	Let $\gamma\in\AA_{G/K}$, $\lambda\in\Lambda$, and write $HC(\gamma)=p_{\gamma}\xi$ as in Remark \ref{poly_from_HC}.  Then the following are equivalent:
	\begin{enumerate}
		\item $HC(\gamma):L_{\lambda}\to k$ is nonzero;
		\item $\UU\k'L_{\lambda}$ contains a copy of $I_{K'}(k)$ and $p_{\gamma}(\lambda)\neq0$;
		\item $L_{\lambda}$ is a projective $\a_{\ol{1}}$-module and $p_{\gamma}(\lambda)\neq0$;
		\item $\lambda\in U_{reg}$ and $p_{\gamma}(\lambda)\neq0$.
	\end{enumerate}
\end{cor}

\subsection{Relationship between $G/K$ and $G/K'$}  We continue to suppose that $(\g,\k)$ satisfies the Iwasawa decomposition.

\begin{prop}\label{injective_summand_prop}
	For $\lambda\in\Lambda^+$, suppose that 
	\begin{enumerate}
		\item $\langle K'\cdot L_{\lambda}\rangle$ contains a copy of $I_{K'}(k)$; and 
		\item $\langle G\cdot L_{\lambda}\rangle\cong L(\lambda)$ is an irreducible $G$-module.
	\end{enumerate} 
	Then $I_G(L(\lambda))$ is a submodule of $k[G]^{-\ber_{\k'}}$.  
\end{prop}
By $k[G]^{-\ber_{\k'}}$  we denote the elements of $k[G]$ which are semi-invariants of weight $-\ber_{\k'}$ under the action of $K'$ on $k[G]$ by right translation.  This space is equal to the sections of the induced bundle $\Ind_{K'}^{G}\Ber(\k')$, and if $\Ber(\k')$ is trivial then $k[G]^{-\ber_{\k'}}=k[G/K']$.
\begin{proof}
	Since $I_{K'}(k)\cong P_{K'}(\Ber(\k'))$, $I_{K'}(k)$ has a morphism $\varphi:I_{K'}(k)\to\Ber(\k')$ determined by the projection onto its head.  Since $I_{K'}(k)$ splits off $L(\lambda)$ as a $K'$-module, it also must split off $I_G(L(\lambda))$ as a $K'$-module.  Therefore we may extend $\varphi$ to a $K'$-equivariant morphism $\phi:I_G(L(\lambda))\to\Ber(\k')$.  By construction $\phi$ is non-zero on $L(\lambda)$, the socle of $I_G(L(\lambda))$. Thus by Frobenius reciprocity, $\phi$ defines an injective morphism of $G$-modules 
	\[
	\Phi:I_G(L(\lambda))\to k[G]^{\ber_{\k'}}.
	\]	
\end{proof}
\begin{remark}
	The above theorem is especially useful when $G/K\cong G/K'$, as it helps to determine the structure of $k[G/K]$ as a $G$-module.  We will see one nice application of this in the next section.
\end{remark}

\subsection{The distribution $\operatorname{ev}_{eK}v_{\k'}$}
Recall that $\operatorname{ev}_{eK}v_{\k'}=v_{\k'}\cdot \operatorname{ev}_{eK}$ is a ghost distribution, and is the one of minimal degree.	  Observe that we have a $K'$-equivariant embedding
\[
i:K'/K_0\to G/K.
\]
This induces a $K'$-equivariant surjective morphism on functions
\[
i^*: k[G/K]\to k[K'/K_0].
\]
Write $\operatorname{ev}_{eK}$ for the evaluation distribution at $eK$ on $G/K$.  Then on $K'$-distributions we obtain a map which sends
\[
v_{\k'}\mapsto v_{\k'}\cdot\operatorname{ev}_{eK}=\operatorname{ev}_{eK}v_{\k'}.
\]
Thus for $\lambda\in\Lambda^+$ we have
\[
(\operatorname{ev}_{eK}v_{\k'}) f_{\lambda}=v_{\k'}i^*(f_\lambda)=\operatorname{ev}_{eK_0}\circ D_{\k'}\circ i^*(f_{\lambda}).
\]
It follows that:
\begin{lemma}
	For $\lambda\in\Lambda^+\cap U_{reg}$, $p_{\operatorname{ev}_{eK}v_{\k'}}(\lambda)\neq0$ if and only if  $\langle K'\cdot i^* f_{\lambda}\rangle$ contains $I_{K'}(k)$.	In particular, if $\langle K'\cdot f_{\lambda}\rangle$ contains a copy of $I_{K'}(k)$, then $p_{\operatorname{ev}_{eK}v_{\k'}}(\lambda)\neq0$ if and only if the $K'$-invariant in the copy of $I_{K'}(k)$ is non-zero at $eK$.  
\end{lemma}

The following property of $v_{\k'}\cdot\operatorname{ev}_{eK}$ will be important in the next section.
\begin{cor}\label{non_vanishing_cor}
	Suppose that whenever a $B$-irreducible submodule $L_{\lambda}\sub k[G/K]$ has $\UU\k'L_{\lambda}$ containing $I_{K'}(k)$, the $K'$-invariant is non-vanishing at $eK$.  Then
	\begin{enumerate}
		\item $\UU\k'L_{\lambda}$ contains $I_{K'}(k)$ if and only if $L_{\lambda}$ is projective over $\a_{\ol{1}}$ and $p_{G/K}(\lambda)\neq0$; and
		\item if $HC(\operatorname{ev}_{eK}v_{\k'}):L_{\lambda}\to k$ is zero, then $HC(\gamma):L_{\lambda}\to k$ is zero for all $\gamma\in\AA_{G/K}$.
	\end{enumerate}
\end{cor}

\section{The Group Case $G\times G/G$}\label{group_case_section}

We now consider the case of the supersymmetric space $G\times G/G$, where $G$ is embedded diagonally.  This space is isomorphic to $G$ with the $G\times G$ action given by left and right translation.  Since $G$ is assumed to be quasireductive, it must be connected in particular, so we can work with the Lie superalgebra without losing much.  The involution we take in this case, $\theta$, is given on the Lie superalgebra $\g\times\g$ by $\theta(x,y)=(y,x)$.   

\begin{lemma}\label{Ug_dist}
	We have a $\g\times\g$-module isomorphism 
	\[
	\UU\g\cong \Dist(G\times G/G,eG)=\UU(\g\times\g)/\UU(\g\times\g)\g
	\]
	given by
	\[
	u\mapsto u\otimes 1.
	\]
\end{lemma}

\begin{proof}
Let $(v_1,v_2)\in\g\times\g$ and $u\in\UU\g$.  We see that $(v_1,v_2)\cdot u=v_1u-(-1)^{\ol{u}\ol{v_2}}uv_2$, and this will map to (before modding out by $\UU(\g\times\g)\g$)
\begin{eqnarray*}
 v_1u\otimes 1-(-1)^{\ol{u}\ol{v_2}}uv_2\otimes 1& = &v_1u\otimes 1+(-1)^{\ol{v_2}\ol{u}}u\otimes v_2-(-1)^{\ol{v_2}\ol{u}}uv_2\otimes 1-(-1)^{\ol{v_2}\ol{u}}u\otimes v_2\\
                                          & = &(v_1\otimes 1+1\otimes v_2)(u\otimes 1)-(-1)^{\ol{v_2}\ol{u}}(u\otimes 1)(v_2\otimes 1+1\otimes v_2)
\end{eqnarray*}
which shows the $\g\times\g$-equivariance of the map.  It is an isomorphism by the PBW theorem.
\end{proof}

From now on we work with $\UU\g$ as our space of distributions.  Observe that in this case,
\[
\g'=\{(u,\delta(u)):u\in\g\}\cong\g,
\]
and so we can identify $\g'$ with $\g$ as a Lie superalgebra.  With this setup, the action of $\g'$ on $\Dist(G\times G/G,eG)$ is given by, for $u\in\g$ and $v\in\UU\g$,
\[
u\cdot v=uv-(-1)^{\ol{u}\ol{v}+\ol{u}}vu.
\]
This is exactly Gorelik's twisted adjoint action defined in \cite{gorelik2000ghost}.  There, she proved algebraically that $\UU\g$ is an induced module from $\g_{\ol{0}}$ under this action.  However this follows from our geometric perspective via Proposition \ref{symm_space_dist}.  Gorelik defined $\AA\sub\UU\g$, the anticentre of $\UU\g$, to be the $\g'$-invariant distributions on $G$.  However please note Remark \ref{A_notation_caution} regarding notation.

\subsection{Structure of $k[G]$ as a $G\times G$-module}  We now assume that $G$ is quasireductive.  Choose a Cartan subalgebra $\h\sub\g$.   Then $\h\times\h$ is a Cartan subalgebra of $\g\times\g$ such that $\a=\{(h,-h):h\in\h\}$ is a Cartan subspace of $\p$.

If we choose a Borel subalgebra $\b$ of $\g$ containing $\h$, then $\b^-\times\b^+$ becomes an Iwasawa Borel subalgebra of $\g\times\g$.  Again for this choice of Borel subalgebra, both $(\g\times\g,\g)$ and $(\g\times\g,\g')$ admit an Iwasawa decomposition for any such choice of Borel subalgebra as in Section \ref{group_case_section}.  Further, if $L_{\lambda}\sub k[G]$ is an irreducible $B^-\times B^+$ submodule, then 
\[
\UU(\g\times\g)L_{\lambda}=\UU\g L_{\lambda}=\UU\g'L_{\lambda}.
\]

We again have $\Lambda^+=\{(-\lambda,\lambda):\lambda\text{ is a }B\text{-dominant weight}\}$.  By abuse of notation we will also write $\Lambda^+$ for the set of $B$-dominant weights in $\h_{\ol{0}}^*$.

\begin{lemma}\label{two_possibilities}
	Let $L(\lambda):=L_{B}(\lambda)$ be the irreducible $G$-module of highest weight $\lambda\in\Lambda^+$. Then one of the following two must occur:
	\begin{enumerate}
		\item $L(\lambda)^*\boxtimes L(\lambda)$ is an irreducible $G\times G$-module and admits a unique even $G'$-invariant.
		\item $L(\lambda)^*\boxtimes L(\lambda)=L\oplus\Pi L$ is a sum the two irreducible $G\times G$-modules $L$ and $\Pi L$ which are non-isomorphic parity shifts of one another.  In this case, both $L$ and $\Pi L$ admit a unique $G'$-invariant, one even and one odd.
	\end{enumerate}
\end{lemma}
\begin{proof}
	If $L(\lambda)\not\cong\Pi L(\lambda)$ then the first case happens.  Otherwise the second case happens.
\end{proof}
\begin{definition}
	For $\lambda\in\Lambda^+$, define $d(\lambda)\in\{0,1\}$ to be $0$ if $L(\lambda)\not\cong\Pi L(\lambda)$, and $1$ otherwise.  If $d(\lambda)=1$, set $\frac{1}{2}L(\lambda)^*\boxtimes L(\lambda)$ to be the irreducible $G\times G$-submodule of $L(\lambda)^*\boxtimes L(\lambda)$ which contains an even $G'$-invariant.
\end{definition}

\begin{prop}\label{appendix_tr_nonzero}
	We have the following decomposition:
	\[
	\operatorname{soc}k[G]=\bigoplus\limits_{\lambda\in\Lambda^+}\frac{1}{2^{d(\lambda)}}L(\lambda)^*\boxtimes L(\lambda).
	\]
\end{prop}
\begin{proof}
	The irreducible $G\times G$ modules are all of the form $\frac{1}{2^{d(\lambda)d(\mu)}}L(\lambda)\boxtimes L(\mu)$, where $\lambda,\mu\in\Lambda^+$.  By Frobenius reciprocity, this admits an even $G\times G$-equivariant morphism into $k[G]$ if and only if $\lambda=-\mu$ and the unique $G$-invariant of it is even, implying the same of about the unique $G'$-invariant.
\end{proof}
	
\begin{lemma}
	Let $L$ be a simple $G$-module.  Then $\End(L)$, as a $G\times G$-module, admits a unique, up to scalar, even $G'$-invariant given by $\delta_L$, the parity involution.  Consequently, $\tr_L\in\End(L)^*$ defines a nonzero $G'$-invariant function on $G$.
\end{lemma}

\begin{proof}
	Clearly $\delta_L$ is $\g_{\ol{0}}$-invariant.  For $u\in\g_1$, we see that
	\[
	((u,-u)\delta)(v)=u\delta(v)+\delta(uv)=(-1)^{\ol{v}}uv+(-1)^{\ol{u}+\ol{v}}uv=0
	\]
	so $\gamma$ is $\g'$-invariant.  Now as a tensor, $\delta_L=\sum\limits_{i}(-1)^{\ol{e_i}}e_i\otimes\varphi_i$, where $\varphi_i(e_i)=1$, so after applying the braiding to switch the order of tensors, we obtain the trace in $\End(L)^*$.
\end{proof}

\begin{cor}\label{tr_nonzero}
	If $L$ is a simple $G$-module, $\frac{1}{2^{d(\lambda)}}L^*\boxtimes L\sub k[G]$ has a unique up to scalar $G'$-invariant, given by $\tr_L$.  In particular, $\tr_L(eG)=\dim L\neq0$, so that the hypotheses of Corollary \ref{non_vanishing_cor} apply.  
\end{cor}

\begin{remark}\label{appendix_remark_proj}
	We have shown that the irreducible $B^-\times B^+$-submodules of $k[G]$ are given by $\frac{1}{2^{d(\lambda)}}L(\lambda)_{-\lambda}^*\boxtimes L(\lambda)_{\lambda}$ and 
	\[
	\UU\g'\left(\frac{1}{2^{d(\lambda)}}L(\lambda)_{-\lambda}^*\boxtimes L(\lambda)_{\lambda}\right)=\frac{1}{2^{d(\lambda)}}L(\lambda)^*\boxtimes L(\lambda).
	\]
	
	Notice that in this setting $\dim\t_{\ol{1}}=\dim\a_{\ol{1}}$, so by Lemma \ref{appendix_ker_lemma} if $\frac{1}{2^{d(\lambda)}}L(\lambda)^*\boxtimes L(\lambda)$ contains $I_{G'}(k)$ then necessarily $\frac{1}{2^{d(\lambda)}}L(\lambda)_{-\lambda}^*\boxtimes L(\lambda)_{\lambda}$ is a projective $\h\times\h$-module, and it is not difficult to show this is equivalent to $L(\lambda)_{\lambda}$ being a projective $\h$-module.
\end{remark} 

\subsection{Projectivity criteria for irreducible modules} For this section we work with a fixed Borel subalgebra $\b$ of $\g$, which as stated above determines a fixed Iwasawa Borel subalgebra of $\g\times\g$.  

We observe that $\id\times\delta$ defines an automorphism of $G\times G$ which takes $G$ to $G'$ and vice-versa.  In particular it defines an isomorphism 
\[
G\times G/G'\to G\times G/G
\]
which is $\id\times\delta$-equivariant, meaning that the pullback morphism defines a $G\times G$-equivariant isomorphism
\[
k[G]=k[G\times G/G]\to k[G\times G/G']^{\id\times\delta}.
\]
Here we use the notation $V^{\phi}$ for the $\phi$ twist of a $G$-module $V$, where $\phi$ is an automorphism of $G$. Notice that since $\id\times\delta$ is the identity on $G_0\times G_0$, we have that $(L^*\boxtimes L)^{\id\times\delta}\cong L^*\boxtimes L$ for a simple $G$-module $L$, by highest weight theory.

\begin{lemma}\label{proj_lemma_appendix}
	$L(\lambda)$ is a projective $G$-module if and only if $\frac{1}{2^{d(\lambda)}}L(\lambda)^*\boxtimes L(\lambda)$ contains a copy of $I_{G'}(k)$.
\end{lemma}
\begin{proof}
	The forward direction is clear.  
	
	Conversely, if this module contains $I_{G'}(k)$ then we know by Proposition \ref{injective_summand_prop} that 
	\[
	I_{G\times G}\left(\frac{1}{2^{d(\lambda)}}L(\lambda)^*\boxtimes L(\lambda)\right)
	\]
	is a $G\times G$ submodule of $k[G]$, using our above isomorphism $G\times G/G\cong G\times G/G'$.  Further, $L(\lambda)_{\lambda}$ must be a projective $H$-module by Remark \ref{appendix_remark_proj}, so that $L(\lambda)$ has no self-extensions.  If there was a non-trivial extension $V$ between $L(\lambda)$ and $L(\mu)$, where $\mu\neq\lambda$, then the matrix coefficients morphism would induce a $G\times G$ morphism
	\[
	V^*\boxtimes V\to k[G]
	\]
	such that the image is an indecomposable module with socle 
	\[
	\frac{1}{2^{d(\lambda)}}L(\lambda)^*\boxtimes L(\lambda)\oplus \frac{1}{2^{d(\mu)}}L(\mu)^*\boxtimes L(\mu).
	\]
	However $I_{G\times G}(\frac{1}{2^{d(\lambda)}}L(\lambda)^*\boxtimes L(\lambda))$ splits off $k[G]$, so this cannot happen.  It follows that $L(\lambda)$ cannot have nontrivial extensions with any modules, and is therefore projective.
\end{proof}

Let us look at the Harish-Chandra homomorphism.  Translating to the enveloping algebra, it defines a homomorphism
\[
HC:\UU\g\to\UU\h.
\]
Note that $\h$ is no longer abelian in general.  However this induces a morphism
\[
HC:\AA_{G\times G/G}\to \AA_{H\times H/H}.
\]
Since $H_0$ is a central subgroup, $\Lambda^{top}\h_{\ol{1}}$ is a trivial $H_0$-module, and so $\AA_{H\times H/H}=(\UU\h)^{\h'}$.  Write $T_{\h}:=\ad'(v_{\h})(1)$.  Then since $\h_{\ol{0}}$ is central, we clearly have
\[
\AA_{H\times H/H}=S(\h_{\ol{0}})T_{\h}.
\]
Since $HC(\ad'(v_{\g})(1))\in (\UU\h)^{\h'}$, we may write $HC(\ad'(v_{\g})(1))=p_{1}T_{\h}$.

For each $\lambda\in\h_{\ol{0}}^*$ we have a bilinear form $(-,-)_{\lambda}$ on $\h_{\ol{1}}$, in other words we have a linear map $(-,-):\h_{\ol{0}}^*\to S^2\h_{\ol{1}}^*$.  If we choose a basis for $\h_{\ol{1}}$, then we obtain a map $S^2\h_{\ol{1}}^*\to k$ given by taking the determinant of the corresponding bilinear form.  The composition defines a degree $\dim\h_{\ol{1}}$-degree polynomial $b_H\in S(\h_{\ol{0}})$.  Note that $b_{H}$ is well-defined only up to scalar.  Further, $b_H(\lambda)\neq0$ if and only if the irreducible $H$-module of weight $\lambda$ is projective.  
\begin{definition}
	Define the projectivity polynomial of $G$ with respect to $B$ and $H$ to be $p_{G,B}:=p_1b_H$.
\end{definition}
\begin{remark}
	We note that $p_{1}$ corresponds to the polynomial $p_{v_{\g'}\cdot\operatorname{ev}_{eK}}$ introduced in Section \ref{sec HC}, while the existence of $b_H$ follows from Remark \ref{rem dims equal}.
\end{remark}

Notice that if $G=H$ we have $p_{G,B}=b_H$.  The following theorem justifies the name.

\begin{thm}\label{appendix_proj_poly_main_thm}
	Let $G$ be a quasireductive supergroup with the Cartan subgroup $H$ and Borel subgroup $B$ containing $H$. Then for a $B$-dominant weight $\lambda$, $p_{G,B}(\lambda)\neq0$ if and only if $L_{B}(\lambda)$ is projective.  Further $p_{G,B}$ is a polynomial of degree at most $\dim\b_{\ol{1}}$.
\end{thm}

\begin{proof}
	If $p_{G,B}(\lambda)\neq0$, then $b_{H}(\lambda)\neq0$ which implies that $\frac{1}{2^{d(\lambda)}}L(\lambda)^*_{-\lambda}\boxtimes L(\lambda)_{\lambda}$ is projective.  Since $p_{1}(\lambda)\neq0$, by Corollary \ref{p_gamma_nonzero_condition} $\frac{1}{2^{d(\lambda)}}L(\lambda)^*\boxtimes L(\lambda)$ contains $I_{G'}(k)$ and so is projective by Lemma \ref{proj_lemma_appendix}.
	
	Conversely if $L(\lambda)$ is projective then $L(\lambda)_{\lambda}$ is projective over $H$ so that $b_H(\lambda)\neq0$.  Further, $\frac{1}{2^{d(\lambda)}}L(\lambda)^*_{-\lambda}\boxtimes L(\lambda)_{\lambda}$ contains $I_{G'}(k)$, and so by Corollary \ref{p_gamma_nonzero_condition} and Proposition \ref{appendix_tr_nonzero} $p_{1}(\lambda)\neq0$.  It follows that $p_{G,B}(\lambda)\neq0$.
\end{proof}

Define $T_{\g}:=\ad'(v_{\g})(1)\in\UU\g$.  Then if $L$ is an irreducible $G$-module, $T_{\g}$ acts on $L$ by a twisted-invariant operator, and thus is either equal, up to a scalar, to $\delta_{L}$ if $T_{\g}$ is even, or $\delta_L\circ\sigma$ if $T_{\g}$ is odd, where $\sigma$ is an odd $G$-equivariant automorphism of $L$.  In particular it either acts by a linear automorphism or by 0.  By Theorem \ref{appendix_proj_poly_main_thm}, we have:
\begin{cor}
	If $L$ is an irreducible $G$-module, then $T_{\g}$ acts by an automorphism on $L$ if and only if $L$ is projective, and otherwise $T_{\g}$ acts by $0$.
\end{cor}

The following conjecture has been verified for the case of $G=H$, i.e. $G_0$ is a central torus in $G$.
\begin{conj}
	Let $M$ be an indecomposable $G$-module.  If $M$ is not projective, then $T_{\g}\cdot M=0$.  If $M$ is projective, then $T_{\g}$ annihilates the radical of $M$ and takes its head to its tail.
\end{conj}

\begin{remark}
	A universal formula for $T_{\g}$ has been given in \cite{duflo2007symmetric}, where they relate this operator to the Jacobian of the exponential map between $\g/\g_{\ol{0}}$ and $G/G_0$.  Their work also extends to arbitrary supersymmetric spaces.
\end{remark}

In the following corollary we write $\ZZ$ for the center of $\UU\g$.
\begin{cor}
	If $HC(\ZZ)=k$ (e.g.\ if $\ZZ=k$) and $\g\neq0$, then no finite-dimensional irreducible $G$-modules are projective.
\end{cor}
\begin{proof}
	If there exists a finite-dimensional irreducible projective $G$-module, then $\Ber(\g)$ must be trivial and thus $\AA_{G\times G/G}=\AA$.  Thus $T_{\g}^2$ will thus be a central element which acts nontrivially on a simple finite-dimensional module, so that $HC(T_{\g}^2)$ is a nonzero scalar in $k$. It follows that $T_{\g}$ acts on every simple finite-dimensional module by an automorphism, meaning that $\operatorname{Rep}(G)$ is semisimple.  However by the classification of algebraic supergroups with semisimple representation theory, $HC(\ZZ)$ must be larger than constants (see \cite{sherman2020two}).
\end{proof}

\subsection{Cartan-even supergroups}
\begin{definition}
	We say that a quasireductive supergroup $G$ is Cartan-even if $\h=\h_{\ol{0}}$, i.e. a maximal torus in $G_0$ remains self-centralizing in $G$.
\end{definition}

We now can prove a nice sufficient criteria for when $G$ admits projective irreducible modules.
\begin{thm}\label{sufficient_proj_criteria}
	Suppose that $G$ is a Cartan-even quasireductive supergroup such that the following conditions on its Lie superalgebra, $\g$, hold: 
	\begin{enumerate}
		\item $\dim[\g_{\alpha},\g_{-\alpha}]=1$ for all $\alpha\in\Delta$; and
		\item for all $\alpha\in\Delta$, the pairing 
		\[
		[-,-]:\g_{\alpha}\otimes\g_{-\alpha}\to[\g_{\alpha},\g_{-\alpha}]
		\]
		is nondegenerate.
	\end{enumerate}
	Make a choice of positive roots and thus a Borel subalgebra $\b$ containing $\h$, and let $\Delta_{\ol{1}}^+=\{\alpha_1,\dots,\alpha_n\}$.  Let $h_{\alpha_i}\in[\g_{\alpha_i},\g_{-\alpha_i}]$ be a chosen nonzero element, and write $r_i=\dim\g_{\alpha_i}$. Then up to a scalar we have:
	\[
	p_{G,B}=h_{\alpha_1}^{r_1}\cdots h_{\alpha_n}^{r_n}+l.o.t.
	\]
	In particular $p_{G,B}\neq0$, so that $G$ admits projective irreducible modules.
\end{thm}
\begin{proof}
	We may choose a weight basis of $\g_{\ol{1}}$ as follows.  Let $u_0,\dots,u_m$ be a root vector basis of $\n_{\ol{1}}$, where $u_i$ has weight $\beta_i$, and the $\beta_i$'s need not be distinct. Then there exists a root vector basis $v_0,\dots,v_m$ of $\n^-_{\ol{1}}$ such that $v_i$ has weight $-\beta_i$ and $[u_i,v_j]=\delta_{ij}h_{\beta_i}$ whenever $\beta_i=\beta_j$.  Write $u_I$ and $v_I$ for the corresponding ordered products when $I\sub\{0,\dots,m\}$, and set $u_{\emptyset}=v_{\emptyset}=1$.   For a subset $J\sub\{0,\dots,m\}$ with $J=\{j_1<\dots<j_l\}$ we set
	\[
	\widetilde{v_J}=v_{j_l}\cdots v_{j_1}.
	\]
	Then we may take as a basis for $\UU\g/\UU\g\g_{\ol{0}}$ the elements $u_I\widetilde{v_J}$ for $I,J\sub\{0,\dots,m\}$.   Therefore we may write
	\[
	v_{\g}=u_0\cdots u_mv_m\cdots v_0+v',
	\]
	where $v'\in\FF^{\dim\g_{\ol{1}}-2}$.  Thus 
	\[
	\ad'(v_{\g})(1)=\ad'(u_0\cdots u_mv_m\cdots v_0)(1)+\ad'(v')(1).
	\]
	By Lemma \ref{degree_bound} $HC(\ad'(v')(1))\leq m$, so it suffices to show that (up to a scalar)
	\[
	\ad'(u_0\cdots u_mv_m\cdots v_0)(1)=h_{\beta_0}\cdots h_{\beta_m}+l.o.t.
	\]
	Observe that
	\[
	\ad'(v_n\cdots v_0)(1)=\sum\limits_{I\sub\{1,\dots,m\}}(-1)^{i_1+\dots+i_j}\widetilde{v_{I^c}}v_I,
	\]
	(note the $\{1,\dots,m\}$ in the sum is not a typo) where $I=\{i_1,\dots,i_j\}$ and $I^c$ denotes the complement of $I$.
	Now if we apply $\ad'(u_0\cdots u_m)$ to $\ad'(v_m\cdots v_0)(1)$, the only term that does not vanish under the Harish-Chandra morphism is
	\[
	\sum\limits_{I\sub\{1,\dots,m\}}(-1)^{i_1+\dots+i_j}u_0\cdots u_m\widetilde{v_{I^c}}v_I.
	\]
	Now it suffices to show that 
	\[
	HC((-1)^{i_1+\dots+i_j}u_0\cdots u_m\widetilde{v_{I^c}}v_I)=h_{\beta_0}\cdots h_{\beta_m}+l.o.t.
	\]
	The proof of this is inductive.  If $m\notin I$, then we may write
	\[
	u_0\cdots u_m\widetilde{v_{I^c}}v_I=u_0\cdots u_mv_m\widetilde{v_{I^c\setminus\{m\}}}v_{I}.
	\]
	This is equal to (after removing the term in $v_m\UU\g$)
	\begin{eqnarray*}
		\sum\limits_{j} (-1)^{m-j}u_0\cdots [u_j,v_m]\cdots u_m\widetilde{v_{I^c\setminus\{n\}}}v_{I\setminus\{m\}}& + &u_0\cdots u_{m-1}\widetilde{v_{I^c\setminus\{m\}}}v_{I}h_{\beta_m}\\
		& + &cu_0\cdots u_{m-1}\widetilde{v_{I^c\setminus\{m\}}}v_{I}
	\end{eqnarray*}
	In the last summand $c$ is the constant given by $(-\beta_0-\dots-\beta_{m-1})(h_{\beta_m})$, but its precise value does not matter as this term will have Harish-Chandra projection of degree $m$ or less.   For the first summand, let $e_{j,m}=[u_j,v_m]\in\g_{\ol{0}}$, and suppose that $e_{j,m}$ lies in $\n^-$ (note  that it cannot lie in $\h$ by our choice of basis).  The argument for when it lies in $\n^+$ is entirely analogous.  Then we have (after removing the term in $e_{j,m}\UU\g$.)
	\[
	u_0\cdots e_{j,m}\cdots u_m\widetilde{v_{I^c\setminus\{m\}}}v_{I}=\sum\limits_lu_0\cdots[u_l,e_{j,m}]\cdots\hat{u_j}\cdots u_m\widetilde{v_{I^c\setminus\{m\}}}v_{I}
	\]
	Thus we end up with an element in $\FF^{2m}$ so that its Harish-Chandra projection is of degree at most $m$, and so does not contribute to the top degree.
	
	Now suppose instead $m\in I$, so that we can write
	\[
	u_0\cdots u_m\widetilde{v_{I^c}}v_I=u_0\cdots u_m\widetilde{v_{I^c}}v_{I\setminus\{m\}}v_m.
	\]
	Write $I=\{i_1,\dots,i_s,m\}$ and $I^c=\{j_1,\dots,j_t\}$.  Then the above expression is equal to
	\begin{eqnarray*}
		(-1)^{n}u_0\cdots u_{m-1}\widetilde{v_{I^c}}v_{I\setminus\{m\}}h_{\beta_m}& + &\sum\limits_{p}u_0\cdots u_{m-1}v_{j_t}\cdots[u_m,v_{j_p}]\cdots v_{j_t}v_{I}\\
		& + &\sum\limits_{q}u_0\cdots u_{m-1}\widetilde{v_{I^c}}v_{i_1}\cdots [u_m,v_{i_q}]\cdots v_n.
	\end{eqnarray*}
	Now in the last two summands we may apply a similar argument as above to prove that their Harish-Chandra projections are of degree less than or equal to $m$.  Therefore, we have shown that the following two elements have the same top degree term:
	\[
	HC((-1)^{i_1+\dots+i_j}u_0\cdots u_m\widetilde{v_{I^c}}v_I), \ \ \ \ HC((-1)^{j_1+\cdot+j_l}u_0\cdots u_{m-1}\widetilde{v_{J^c}}v_J)h_{\alpha_m},
	\]
	where $\{j_1,\dots,j_l\}=J=I\cap\{0,\dots,m-1\}$ and the complement of $J$ is taken in $\{0,\dots,m-1\}$.  Now we may apply an induction argument to show that
	\[
	HC((-1)^{i_1+\dots+i_j}u_0\cdots u_m\widetilde{v_{I^c}}v_I)=h_{\beta_0}\cdots h_{\beta_m}+l.o.t.
	\]
	which completes the proof.
\end{proof}
We say a Lie superalgebra is quadratic if it admits a nondegenerate, invariant, and even supersymmetric form $(-,-)$.  The following is straightforward to show from the properties of being quadratic.
\begin{cor}
	Let $G$ be a Cartan-even quasireductive supergroup whose Lie superalgebra is quadratic.  Then the conditions of Theorem \ref{sufficient_proj_criteria} hold.  In particular $G$ admits projective irreducible modules.
\end{cor}

\subsection{$\g$ basic classical} We now additionally assume that $\g$ is almost simple and basic classical, by which we mean it is either basic simple, $\s\l(n|n)$, or $\g\l(m|n)$.  We have the following two results.

\begin{cor}
	$\prod\limits_{\alpha\in\Delta^+_{\ol{1}},(\alpha,\alpha)=0}(h_{\alpha}+(\rho,\alpha))$ divides $p_{G,B}$ and thus divides $HC(\gamma)$ for any $\gamma\in\Dist(G/K,eK)^{\k'}$.
\end{cor}
\begin{proof}
The first statement follows from Theorem \ref{appendix_proj_poly_main_thm} and a density argument, and the second statement follows from Corollary \ref{non_vanishing_cor}.
\end{proof}

The above result recovers part of Gorelik's result in \cite{gorelik2000ghost}, where she proved that
\[
HC(\AA)=S(\h)^{W_.}\prod\limits_{\alpha\in\Delta^+_{\ol{1}}}(h_{\alpha}+(\rho,\alpha)).
\]
Here $S(\h)^{W_.}$ denotes the $\rho$-shifted $W$-invariant polynomials in $\h$.  Gorelik proved this result by studying the action of $\AA$ on Verma modules of $\g$, and using known results on highest weight vectors in Verma modules.  Thus far the author does not know how to reprove her results from a purely geometric standpoint.  In particular the invariance under an even Weyl group remains elusive.

\section{Full ghost centre}\label{ghost_centre_section}

As explained previously, in \cite{gorelik2000ghost} an algebra called the ghost centre $\tilde{\ZZ}\sub\UU\g$ was defined as
\[
\tilde{\ZZ}:=\ZZ+\AA,
\]
where $\AA=(\UU\g)^{\g'}$.  One can check that for $a,b\in\AA$ we have $ab\in\ZZ$, so this indeed forms an algebra.  We seek to expand this algebra to what we will call the full ghost centre, $\ZZ_{full}$, of $\g$.  It is a subalgebra of $\UU\g$ which we now describe.  

\subsection{Twisted adjoint actions}Let $\Aut(\g,\g_{\ol{0}})$ denote the set of automorphisms of $\g$ which fix $\g_{\ol{0}}$ pointwise.  For $\phi\in\Aut(\g,\g_{\ol{0}})$, define the $\phi$-twisted adjoint action on $\UU\g$ by
\[
\ad_{\phi}(u)(v)=uv-(-1)^{\ol{u}\ol{v}}v\phi(u)
\]
where $u\in\g$, $v\in\UU\g$.
\begin{definition}
	 Define $\AA_{\phi}$ to be the $\ad_{\phi}$-invariant elements of $\UU\g$.
\end{definition}  
Observe that $\AA_{\id}=\ZZ$ and $\AA_{\delta}=\AA$, where $\delta(u)=(-1)^{\ol{u}}u$.  

Another way to understand this action is as follows: consider the subalgebra $\g_{\phi}$ of $\g\times\g$ given by $\{(u,\phi(u)):u\in\g\}$.  Then $\g_{\phi}\cong\g$, and its even part is $(\g_{\phi})_{\ol{0}}=\{(u,u):u\in\g_{\ol{0}}\}$.  Then the action of $\g_{\phi}$ on $\Dist(G\times G/G,eG)\cong\UU\g$ exactly induces the $\phi$-twisted adjoint action.

Let $G_{\phi}\sub G\times G$ be the connected subgroup which integrates $\g_{\phi}$.  
\begin{lemma}\label{dist_ind_twisted_action}
Suppose that $\phi(u)=u$ implies $u\in\g_{\ol{0}}$.  Then $\UU\g\cong\Ind_{\g_{\ol{0}}}^{\g}\UU\g_{\ol{0}}$ with respect to the $\phi$-twisted action.  In particular if $\Ber(\g_{\ol{1}})$ is trivial, then there is a natural isomorphism of vector spaces $\ZZ(\UU\g_{\ol{0}})\to \AA_{\phi}$ given by
\[
z\mapsto\ad_{\phi}(v_{\g})(z).
\]
\end{lemma}

\begin{proof}
		If $\phi(u)=u$ implies $u\in\g_{\ol{0}}$, then $\g_{\phi}\to (T_{eG}(G\times G/G))_{\ol{1}}$ is an isomorphism.  Therefore, Corollary \ref{dist_iso} applies.
\end{proof}

Observe that $\id\times\phi$ defines an isomorphism of $G\times G$ taking $G$ to $G_{\phi}$, and thus an isomorphism
\[
G\times G/G\to G\times G/G_{\phi}
\]
which is $\id\times\phi$-equivariant.  Since $\id\times\phi$ is the identity on $\g_{\ol{0}}\times\g_{\ol{0}}$, we obtain the following generalization of Theorem \ref{appendix_proj_poly_main_thm}.
\begin{prop}\label{mild_proj_criteria_z_full}
	Suppose that $\Ber(\g)$ is trivial and that $\phi\in\Aut(\g,\g_{\ol{0}})$ satisfies the conditions of Lemma \ref{dist_ind_twisted_action}.  Then for $\gamma\in\AA_{\phi}$ and $\lambda\in\Lambda^+$, $HC(\gamma):L_{\lambda}\to k$ is the zero map if $L(\lambda)$ is not projective.  Further if the unique $\ad_{\phi}$-invariant element of $\frac{1}{2^{d(\lambda)}}L(\lambda)^*\boxtimes L(\lambda)$ always is nonzero at $eG$ for a simple $G$-module $L$, then $HC(\ad_{\phi}(v_{\g})1)(\lambda)\neq0$ if and only if $L(\lambda)$ is projective.
\end{prop}
The unique $\ad_{\phi}$-invariant element of $\frac{1}{2^{d(\lambda)}}L(\lambda)^*\boxtimes L(\lambda)$ can be thought of as a $\phi$-twisted trace function.  We describe these elements for basic type I algebras in Section \ref{twisted_trace_subsec}.

\subsection{The full ghost centre} For $\phi,\psi\in\Aut(\g,\g_{\ol{0}})$, it is easy to check that multiplication in $\UU\g$ induces a morphism 
\[
\AA_{\phi}\otimes\AA_{\psi}\to\AA_{\psi\phi}.
\]
\begin{definition}
	Define the full ghost centre of $\UU\g$ to be the algebra given by
	\[
	\ZZ_{full}:=\sum\limits_{\phi\in\Aut(\g,\g_{\ol{0}})}\AA_{\phi}.
	\]
\end{definition}
Notice  that $\tilde{\ZZ}\sub\ZZ_{full}\sub(\UU\g)^{\g_{\ol{0}}}$.  

\subsection{The case of an abelian Lie superalgebra} 

As a toy case, we consider the above constructions for the abelian superalgebra $\g=k^{m|n}$. Then $\Aut(\g,{\g_{\ol{0}}})\cong GL(n)$.  Let $A\in GL(n)$, and let $\xi_A$ be a nonzero element of $\Lambda^{top}\operatorname{im}(A-\id)$, where $\xi_{\id}=1$.
\begin{lemma}
	$\AA_{A}=\UU\g\xi_A$.
\end{lemma}
\begin{proof}
	For $u\in\UU\g$ and $\eta\in\g$, we see that
	\begin{eqnarray*}
	\ad_{A}(\eta)(u)& = &\eta u-(-1)^{\ol{u}}uA\eta\\
	                       & = &(\eta-A\eta)u.
	\end{eqnarray*}
From here the result is straightforward.
\end{proof}
In this case we see that the ghost centres for different automorphisms in $\Aut(\g,\g_{\ol{0}})$ may overlap in myriad ways.  Now suppose $A\in\Aut(\g,\g_{\ol{0}})$ doesn't fix any non-zero odd vectors.  Then under the twisted adjoint action determined by $A$, $\UU\g$ is isomorphic to a sum of the left (or right) regular representation of $\UU\g_{\ol{1}}$, and we have $\AA_{A}=S(\g_{\ol{0}})\Lambda^{n}\g_{\ol{1}}$.  By Lemma \ref{v_for_almost_odd_abelian}, we have $v_{\g}=\xi_1\dots\xi_n$, where $\xi_1,\dots,\xi_n$ is any basis of $\g_{\ol{1}}$.  Thus if $A$ is fixed point free, the following element must be non-zero:
\[
\ad_{A}(\xi_1\cdots\xi_n)(1)=(1-A)\xi_1\cdots (1-A)\xi_{n}=\det(1-A)\xi_1\cdots\xi_n.
\]
So this is nonzero if and only if $A-1$ is invertible, i.e. $A$ is not an eigenvalue of $1$.

\begin{remark}
	Using the above computation, it is possible to give a purely algebraic proof of Lemma \ref{dist_ind_twisted_action} with the same proof as given in \cite{gorelik2000ghost}.  Namely, given $u_1,\dots,u_n\in\g_{\ol{1}}$, $v\in\UU\g$, and $\phi\in\Aut(\g,\g_{\ol{0}})$, we have
	\[
	\operatorname{gr}(\ad_{\phi}(u_1\cdots u_n)(v))=\ad_{\phi}(\operatorname{gr}(u_1\cdots u_n))(\operatorname{gr}v),	
	\]
	where $\operatorname{gr}$ is the associated graded morphism $\UU\g\to S(\g)$, and we view $S(\g)$ as the enveloping algebra of $\g$ with trivial bracket.  Thus if $L_{0}\sub\UU\g_{\ol{0}}$ is a $\g_{\ol{0}}$-submodule and $\{v_j\}\sub L_{0}$ is a basis such that $\{\operatorname{gr}v_j\}$ is linearly independent in $\UU\g_{\ol{0}}$, then by passing to the associated graded we find that $\UU\g L_{0}\cong\Ind_{\g_{0}}^{\g}L_0$.
\end{remark}

\subsection{$\ZZ_{full}$ for basic classical Lie superalgebras}\label{ghost_centre_basic_case_subsec}  Let $\g$ be either $\g\l(m|n)$, $\s\l(m|n)$, $\p\s\l(n|n)$ for $n>2$, $\o\s\p(m|2n)$, or a basic exceptional simple Lie superalgebra. Recall that such a superalgebra is called type I if it admits a $\Z$-grading $\g=\g_{-1}\oplus\g_0\oplus\g_1$ compatible with the $\Z_2$-grading.  Of these, $\g\l(m|n)$, $\s\l(m|n)$, $\p\s\l(n|n)$ and $\o\s\p(2|2n)$ are the type I algebras. 

\begin{remark}[\textbf{Caution}] We do not consider $\p\s\l(2|2)$ here because Lemma \ref{aut_basic_classical} is false (see section 5.5 of \cite{musson2012lie}).  Further, in Section \ref{subsec_charac_Z_full}, we will not consider $\s\l(n|n)$ or $\p\s\l(n|n)$ due to the lack of an internal grading operator. 
\end{remark}
\begin{lemma}\label{aut_basic_classical}
	$\Aut(\g,\g_{\ol{0}})\cong k^\times$ if $\g$ is of type I, and otherwise $\Aut(\g,\g_{\ol{0}})=\langle\delta\rangle\cong\Z/2$.  In particular $\tilde{\ZZ}=\ZZ_{full}$ for $\g$ not of type I.
\end{lemma}
\begin{proof}
	We refer to \cite{musson2012lie}.  In the type I case, we identify $\Aut(\g,\g_{\ol{0}})\cong k^\times$ via defining for $c\in k^\times$ the automorphism $c\mathbf{1}_{\g_{-1}}\oplus\mathbf{1}_{\g_0}\oplus c^{-1}\mathbf{1}_{\g_{1}}$.
\end{proof}

So for these algebras we only get something new in the type I case.  We study this case now.  Since $\Aut(\g,\g_{\ol{0}})\cong k^\times$, we label the automorphisms by complex numbers $c\in k^\times$ according to the identification given in the proof of Lemma \ref{aut_basic_classical}.

Write $\g=\g_{-1}\oplus\g_0\oplus\g_1$ for the $\Z$-grading on $\g$.  Choose a Cartan subalgebra $\h\sub\g_{\ol{0}}$ and consider a Borel subalgebra $\b=\b_0\oplus\g_1$ where $\b_0$ is a Borel of $\g_0$ containing $\h_{\ol{0}}$.  We call this a standard Borel subalgebra of $\g$.  Then we obtain a Harish-Chandra morphism with respect to this Borel.  Let $W$ denote the Weyl group of $\g_{\ol{0}}$ and consider its action $\rho$-shifted action on $\h$, where $\rho$ is the Weyl vector.

\begin{lemma}\label{Z_grading_hw_module}
	Every $\b$-highest weight module $M$ admits a $\Z$-grading that is compatible with the $\Z$-grading on $\g$.
\end{lemma}
\begin{proof}
	We may set $M_0=M^{\g_1}$, and then define $M_{-i}=\Lambda^i\g_{-1}M_0$.  
\end{proof}
We will use the $\Z$-grading defined in the proof of Lemma \ref{Z_grading_hw_module} many times in what follows.
\begin{thm}\label{HC_Z_full}
	\begin{enumerate}
		\item $\ZZ_{full}$ is a commutative algebra.
		\item 	For $c\neq 1$, $HC$ is injective on $\AA_{c}$, and we have 
		\[
		HC(\AA_{c})=HC(\AA_{-1})=S(\h)^{W}\prod\limits_{\alpha\in\Delta^+_{\ol{1}}}(h_{\alpha}+(\rho,\alpha)).
		\]
	\end{enumerate}
\end{thm}  

\begin{remark}
	In this case we see that $HC(\AA_c)\sub HC(\ZZ)$. See for instance \cite{cheng2012dualities} for a description of the Harish-Chandra image of the centre for $\g\l(m|n)$ and $\o\s\p(m|2n)$.  
\end{remark}
\begin{proof}
	Let 
	\[
	t_{\g}:=\prod\limits_{\alpha\in\Delta_1^+}(h_{\alpha}+(\alpha,\rho)).
	\]
	By Proposition \ref{mild_proj_criteria_z_full}, if $u\in\AA_{c}$ then we must have $HC(u)(\lambda)=0$ for $\lambda$ an atypical dominant integral weight, i.e. $\lambda$ such that $(\lambda+\rho,\alpha)=0$ for some odd root $\alpha$, or equivalently $t_{\g}(\lambda)=0$.  The set of such $\lambda$ are dense in the zero set of $t_{\g}$, and thus we have that $t_\g$ divides $HC(u)$.
	
	To obtain Weyl group invariance, let $\lambda$ be a dominant integral weight, and consider the Verma module $M_{\b}(\lambda)$ of highest weight $\lambda$ and highest weight vector $v$.  Choose a simple even root $\alpha$, and write $f_{\alpha}$ for a root vector of weight $-\alpha$.  Then $f_{\alpha}^{(\lambda,\alpha)+1}v$ will be a highest weight vector of weight $\lambda-((\lambda,\alpha)+1)\alpha=s_{\alpha}(\lambda)$.  Since $f_{\alpha}\in\g_{0}$, 
	\[
	uf_{\alpha}^{(\lambda,\alpha)+1}v=f_{\alpha}^{(\lambda,\alpha)+1}uv=HC(u)(\lambda)f_{\alpha}^{(\lambda,\alpha)+1}v.
	\]
	Thus $HC(u)(\lambda)=HC(u)(s_{\alpha}(\lambda))$.  Since such $\lambda$ are dense, and the reflections in the even simple roots of this Borel subalgebra generate the Weyl group, it follows that $HC(u)\in S(\h)^{W}$.	Since $t_{\g}$ is itself $W$-invariant, it follows that we've shown
	\[
	HC(\AA_{c})\sub S(\h)^{W}t_{\g}.
	\]
	To finish the proof, first observe that $HC$ is injective on $\AA_{c}$, as if $HC(u)=0$ for $u\in\AA_{c}$ then it must act by zero on every Verma module, implying it is zero by \cite{letzter1994complete}.  Now if we apply Lemma \ref{degree_bound}, we can make the same degree arguments as in \cite{gorelik2000ghost} to obtain the equality $HC(\AA_{c})=S(\h)^{W}t$.

\end{proof}
With Theorem \ref{HC_Z_full} we can now completely describe the structure of $\ZZ_{full}$ for type I algebras.  Let $N=\dim\g_1=\dim\g_{\ol{1}}/2$, and let $\zeta_{N+1}\in k$ be a primitive $(N+1)$th root of unity.
\begin{thm}\label{Z_full_description}
	\[
	\ZZ_{full}=\bigoplus\limits_{i=0}^{N}\AA_{\zeta_{N+1}^i}.
	\]	
\end{thm} 
\begin{proof}
	Choose $u\in\AA_{c}$ for $c\neq 1$, and write $p=HC(u)$.  Let $u_i\in\AA_{\zeta_{N+1}^i}$ be the unique element such that $HC(u_i)=p$.  We want to solve the equation
	\[
	u=a_0u_0+\cdots+a_{N}u_{N}=\sum\limits_{i=0}^{N}a_iu_i.
	\]
	Let $M:=M_{\b}(\lambda)$ be the Verma module of highest weight $\lambda$ for a standard Borel $\b$, and write $M_{j}$ for the $j$th graded part according to the $\Z$-grading defined in the proof of Lemma \ref{Z_grading_hw_module}. Then the scalar action of each side of the above equation applied to $M_j$ gives the equation
	\[
	c^jp(\lambda)=\sum\limits_{i=0}^{N}a_i\zeta_{N+1}^{ji}p(\lambda).
	\]
	If $p(\lambda)=0$ then this equation always holds. If $p(\lambda)\neq0$, we divide by it to get the system of equations
	\[
	c^j=\sum\limits_{i=0}^{N}a_i\zeta_{N+1}^{ji}
	\]
	This has a unique solution for $a_0,\dots,a_{N}$ since it is a linear system for the Vandermonde matrix determined by $1,\zeta_N,\dots,\zeta_{N+1}^{N}$.  Since $u-a_0u_0-\dots-a_{N}u_N$ then acts trivially on every Verma module, it is zero in $\UU\g$.  It follows that
	\[
	\ZZ_{full}=\sum\limits_{i=0}^{N}\AA_{\zeta_{N+1}^i}.
	\]  
	The sum is direct by the nonsingularity of the Vandermonde matrix, so we are done.
\end{proof}

\begin{remark}
	If $c\in k^\times$, we have an inclusion $\AA_{c}\to\ZZ_{full}=\bigoplus\limits_{i=0}^{N}\AA_{\zeta_{N+1}^i}$.  Assume that $c\neq 1$, and for $u\in \AA_{c}$ write $p=HC(u)$. Let $u_{i}\in\AA_{\zeta_{N+1}^i}$ such that $HC(u_{i})=p$.  Then by inverting the Vandermonde matrix, we obtain the decomposition 
	\[
	u=\frac{1}{N+1}\sum\limits_{i=0}^{N}\left(\sum\limits_{j=0}^{N}c^j\zeta_{N+1}^{-ij}\right)u_{i}.
	\]
\end{remark}

\subsection{Characterization of $Z_{full}$}\label{subsec_charac_Z_full} We now give an intrinsic description of $\ZZ_{full}$ for the type I algebras $\g\l(m|n)$, $\s\l(m|n)$ for $m\neq n$, and $\o\s\p(2|2n)$.  These are distinguished by their having an internal grading operator, that is an element $h\in\g_{0}$ such that $[h,u]=\deg(u)u$ for $u\in\g$.  Thus we assume for this subsection that $\g$ is one of these superalgebras.  

Given a $\g$-module $V$, recall that $V$ is $\Z$-graded if there exists a $\Z$-grading of $V$ which is compatible with the $\Z$-grading on $\g$.  In Lemma \ref{Z_grading_hw_module} it was shown that a highest weight module for the Borel $\b_{0}\oplus\g_1$ is $\Z$-graded.  In particular all finite-dimensional irreducible representations of $\g$ are $\Z$-graded.

If $V$ is $\Z$-graded, we say an element $u\in\UU\g$ acts on it by a $\Z$-graded constant if $u$ acts by a scalar on each component of the $\Z$-grading.  We now seek to prove the following analogue of corollary 4.4.4 of \cite{gorelik2000ghost}.
\begin{thm}\label{Z_full_charac}
	$\ZZ_{full}$ consists exactly of all elements of $\UU\g$ which act by $\Z$-graded constants on all finite-dimensional irreducible representations of $\g$.
\end{thm}

Our proof follows the same strategy as taken in \cite{gorelik2000ghost}.  Observe that the $\Z$-grading on $\g$ induces a $\Z$-grading on $\UU\g$:
\[
\UU\g=\bigoplus\limits_{n\in\Z}(\UU\g)_{n}.
\]
Then $(\UU\g)_0$ is a subalgebra of $\UU\g$. Let $C$ denote its centralizer, and let $\b=\b_0\oplus\g_1$ denote a standard Borel subalgebra.  Clearly $\ZZ_{full}\sub C$.  We would like the show that $\ZZ_{full}=C$.
\begin{lemma}
	$C\sub(\UU\g)_0$.
\end{lemma}
\begin{proof}
	Recall that $\g$ has an internal grading operator, $h$, in $\g_{\ol{0}}$.  Thus $[h,C]=0$, implying that $C\sub(\UU\g)_0$.
\end{proof}
Now let $c\in C$.
\begin{lemma}
	If $L(\lambda)$ is an irreducible $\b$-highest weight module, then $c$ acts on $L(\lambda)$ by a $\Z$-graded constant.
\end{lemma}
\begin{proof}
	Write $L(\lambda)=\bigoplus\limits_n L(\lambda)_n$ for a $\Z$-grading on $L(\lambda)$.  Then $L(\lambda)_n$ is a $(\UU\g)_0$-module, and it suffices to prove it is irreducible.  If $W\sub L(\lambda)_n$ is a $(\UU\g)_0$-submodule, then $\UU\g W\cap L(\lambda)_n=W$.  But since $L(\lambda)$ is irreducible this forces either $W=L(\lambda)_n$ or $W=0$, so we are done.  
\end{proof}

Let $S\sub\h^*$ denote the collection of weights $\lambda$ such that $M_{\b}(\lambda)=L_{\b}(\lambda)$.  Note that $S$ is a Zariski dense subset of $\h^*$. 

We give all Verma modules for $\g$ with respect to the standard Borel the $\Z$-grading defined in the proof of Lemma \ref{Z_grading_hw_module}.   Then for our fixed $c\in C$, we write $f_i(\lambda)$ for the constant by which $c$ acts on $M(\lambda)_{-i}$, where $f_i:S\to k$ is some function determined by $c$.

\begin{prop}\label{f_i_are_polys}
	The functions $f_i$ are polynomials on $\h^*$, i.e. $f_i\in S(\h)$.
\end{prop}
For this, we need a lemma:
\begin{lemma}\label{one_diml_wt_space}
	Let $1\leq n\leq\dim\g_1$, and let $\alpha_1,\dots,\alpha_n$ be the $n$ smallest distinct positive odd roots of $\g$ with respect to our choice of Borel.  Then $\dim(\UU\n^{-})_{-\alpha_1-\dots-\alpha_n}=1$.
\end{lemma}
\begin{proof}
The dimension of $(\UU\n^{-})_{-\alpha_1-\dots-\alpha_n}$ is given by the number of ways to write $-\alpha_1-\dots-\alpha_n$ as a sum of negative roots of $\g$, where each odd root can show up at most once.  Suppose that we have
\[
-\alpha_1-\dots-\alpha_n=-\beta_1-\dots-\beta_m-r_1\gamma_1-\dots-r_{m'}\gamma_{m'},
\]
where $\beta_1,\dots,\beta_m$ are positive odd roots are $\gamma_1,\dots,\gamma_{m'}$ are positive even roots. Writing again $h\in\g_{0}$ for the internal grading operator, on $\g$, if we apply $h$ to the above equality of roots we learn that $n=m$.  However by our choice of $\alpha_1,\dots,\alpha_n$, this clearly forces $r_i=0$ for all $i$ and $\{\beta_1,\dots,\beta_m\}=\{\alpha_1,\dots,\alpha_n\}$, so we are done.
\end{proof}
\begin{proof}[Proof of Proposition \ref{f_i_are_polys}]
	Observe that $f_0=HC(c)$, so this is a polynomial.  For $1\leq n\leq \dim\g_1$, let $\alpha_1,\dots,\alpha_n$ be the $n$ smallest distinct positive odd roots with root vectors $u_1,\dots,u_n$, and write $v_1,\dots,v_n$ for the root vectors of weight $-\alpha_1,\dots,-\alpha_i$, where we assume $[u_i,v_i]=h_{\alpha_i}$, and $h_{\alpha_i}$ is the coroot of $\alpha_i$.   Let $\lambda\in S$, and write $v_{\lambda}$ for the highest weight vector of $L(\lambda)$.  Then observe that
	\[
	u_1\cdots u_ncv_n\cdots v_1v_{\lambda}=f_n(a)HC(u_1\cdots u_nv_n\cdots v_1)v_{\lambda}.
	\]
	On the other hand,
	\[
	u_1\cdots u_ncv_n\cdots v_1v_{\lambda}=HC(u_1\cdots u_ncv_n\cdots v_1)v_{\lambda}.
	\]
	Thus on $S$ we have an equality of functions:
	\[
	HC(u_1\cdots u_nv_n\cdots v_1)f_n(c)=HC(u_1\cdots u_ncv_n\cdots v_1).
	\]
	Now let $\lambda\in\h^*$ be arbitrary, and write $v_{\lambda}$ for the highest weight vector of $M_{\b}(\lambda)$.  If $HC(u_1\cdots u_nv_n\cdots v_1)(\lambda)=0$, then $u_1\cdots u_nv_n\cdots v_1v_{\lambda}=0$.  However by Lemma \ref{one_diml_wt_space}, $cv_n\cdots v_1v_{\lambda}$ is again a multiple of $v_n\cdots v_1v_{\lambda}$, so in this case we also have $u_1\cdots u_ncv_n\cdots v_1v_{\lambda}=0$, and thus $HC(u_1\cdots u_ncv_n\cdots v_1)=0$.  
	
	Therefore, wherever $HC(u_1\cdots u_nv_n\cdots v_1)$ vanishes, $HC(u_1\cdots u_ncv_n\cdots v_1)$ also vanishes.  Further, $HC(u_1\cdots u_nv_n\cdots v_1)$ will have top degree term given by $h_{\alpha_1}\cdots h_{\alpha_n}$, and thus this polynomial is a product of distinct irreducible polynomials.  These facts together imply it divides $HC(u_1\cdots u_ncv_n\cdots v_1)$ so that 
	\[
	f_n(c)=HC(u_1\cdots u_ncv_n\cdots v_1)/HC(u_1\cdots u_nv_n\cdots v_1)\in S(\h).
	\]
\end{proof}

\begin{prop}
	$C$ acts by $\Z$-graded constants on every Verma module $M(\lambda)$.
\end{prop}
\begin{proof}
	Let $M=\UU\g\otimes_{\UU\b}S(\h)$ denote the universal Verma module, where $\n^+$ acts trivially on $S(\h)$ and $\h$ acts by multiplication.  This module admits a $\Z$-grading given 
	\[
	M_{-i}=\Lambda^i\g_{-1}\UU\g_{0}S(\h).
	\]
	For each $\lambda\in\h^*$, we have a surjective $\g$-equivariant morphism
	\[
	M\to M(\lambda)
	\]
	given by evaluation at $\lambda$ on $S(\h)$.  For $u\in(\UU\n^{-})_i$, consider the element $[c,u]+(HC(c)(\lambda)-f_i(\lambda))u\in\UU\g$, and then consider
	\[
	\left([c,u]+(HC(c)(\lambda)-f_i(\lambda))u\right)\otimes 1\in M.
	\]
	Then we have shown that this element evaluates to $0$ on $S$, a Zariski dense subset of $\h^*$, and therefore it must vanish under every evaluation.  Since 
	\[
	c\cdot uv_{\lambda}=[c,u]v_{\lambda}+HC(c)(\lambda)uv_{\lambda},
	\]  
	it follows that for every $\lambda$, $c$ acts by $f_i(\lambda)$ on $M(\lambda)_{-i}$, so that it acts by $\Z$-graded constants.
\end{proof}

\begin{prop}
	$Z_{full}=C$.
\end{prop}
\begin{proof}
Define polynomials $p_0,\dots,p_{N-1}\in S(\h)$ by:
\[
\begin{bmatrix}
f_0\\ f_1\\\vdots\\f_{N-1}
\end{bmatrix}=p_0\begin{bmatrix}
1\\1\\\vdots\\1
\end{bmatrix}+p_1\begin{bmatrix}
1\\\zeta_N\\\vdots\\\zeta_N^{N-1}
\end{bmatrix}+\dots+p_{N-1}\begin{bmatrix}
1\\\zeta_N^{N-1}\\\vdots\\\zeta_N^{(N-1)(N-1)}
\end{bmatrix}.
\]
Using the same argument as in Theorem \ref{HC_Z_full}, we must have that $p_i\in S(\h)^{W}$ for all $i$. 
 
 Let $\alpha$ be the unique simple odd isotropic root of this Borel subalgebra, and write $f_{\alpha}$ for a root vector of weight $-\alpha$.  Then when $(\lambda,\alpha)=0$, $f_{\alpha}v$ will be a highest weight vector in $M(\lambda)$.  Thus	
 \[
 p_{i}(\lambda-\alpha)=\zeta_{N}^{i}p_{i}(\lambda).
 \]
 For $i=0$ this implies that $p_0(\lambda+r\alpha)=p_0(\lambda)$ for all $r\in k$, and for $i>0$ this forces
 \[
 p_i(\lambda-n\alpha)=\zeta_{N}^{ni}p_{i}(\lambda).
 \] 
 Since $p_i$ is a polynomial, this forces $p_i(\lambda)=0$, so that $p_i$ vanishes on the hyperplane $(\lambda,\alpha)=0$.  
 
 Using the $\rho$-shifted $W$-invariance of these polynomials, these conditions imply that $p_0$ is constant along all hyperplanes of the form $(\lambda+\rho,\alpha)=0$, so that $p_0\in HC(\ZZ)$ by Sergeev's description of $HC(\ZZ)$.  On the other hand we find that for $i>0$, $t_{\g}$ divides $p_i$ (see the proof of Theorem \ref{HC_Z_full}), so that $p_i\in HC(\AA_{-1})$. 
 
Now we may take, for each $i$, $u_{i}\in\AA_{\zeta_N^i}$ such that $HC(u_i)=p_i$ for all $i$, and consider the element:
\[
c-u_0-\cdots-u_{N-1}.
\]
By construction this acts by 0 on all Verma modules, and thus is zero so that $c\in\ZZ_{full}$, finishing the proof.	
\end{proof}

\begin{thm}\label{Z_full_charac_stronger}
	 Let $\g$ be one of $\g\l(m|n)$, $\s\l(m|n)$ for $m\neq n$, or $\o\s\p(2|2n)$.  The following subalgebras of $\UU\g$ agree with each other:
	\begin{enumerate}
		\item $\ZZ_{full}$;
		\item the centralizer of $(\UU\g)_0$;
		\item the center of $(\UU\g)_0$; and
		\item the collection of elements in $\UU\g$ which act by $\Z$-graded constants on every irreducible finite-dimensional representation of $\g$.  
	\end{enumerate}
\end{thm}

\begin{proof}
	The only nontrivial equality is (2)$\iff$(4).  However by \cite{letzter1994complete}, the set of finite-dimensional irreducible representations of $\g$ form a complete set, meaning their collective annihilator is trivial.  Therefore if an element of $\UU\g$ acts by $\Z$-graded constants on every irreducible finite-dimensional representation of $\g$, it commutes with $(\UU\g)_0$ under every such representation, and thus commutes with $(\UU\g)_0$.
\end{proof}

\subsection{Twisted trace functions}\label{twisted_trace_subsec}
	We continue to let $\g$ be a basic algebra of type I as in the previous section, and let $c\in k^\times$.  Then for a simple $G$-module $L$, let $L_0=L^{\g_{1}}\sub L$ be the invariants of $\g_{1}$ so that 
	\[
	L=L_0\oplus \g_{-1}L_0\oplus\Lambda^2\g_{-1}L_0\oplus\cdots\oplus\Lambda^{top}\g_{-1}L_0
	\]
	defines $\Z$-grading on $L$.  Write $L_{-i}=\Lambda^i\g_{-1}L_0$, and define the operator $T_c\in\End(L)$ by declaring that $T_c$ preserves the $\Z$-grading and $T_c$ acts on $L_{-i}$ by the scalar $c^{-i}$.
	
\begin{lemma}
The submodule $L^*\boxtimes L\sub k[G]$ contains a unique $G_{c}$-invariant function $f_c$ such that
\[
f_c(eG)=\sum\limits_{i\geq 0}(-1)^{i}c^{i}\dim L_{-i}.
\]
\end{lemma}
\begin{proof}
The operator $T_c$ defined above is $G_{c^{-1}}$-invariant in $\End(L)$.  If we apply the braiding isomorphism $L\otimes L^*\cong L^*\otimes L$, $T_c$ becomes a $G_c$-invariant element, which we write as $f_c$.  It is now straightforward to check the above formula for $f_c(eG)$.
\end{proof}

\begin{definition}
	Let $L$ be a simple $G$-module with the $\Z$-grading as above.  Define the polynomial $p_L\in \Z[c]$ to be
	\[
	p_L(c)=f_c(eG)=\sum\limits_{i\geq 0}(-1)^{i}c^{i}\dim L_{-i}.
	\]
\end{definition}
Observe that $0\leq \deg p_{L}\leq \dim\g_{-1}$.  
\begin{lemma}\label{p_L_projective}
	If $L$ is projective, then $p_L(c)\neq0$ if and only if $c\neq1$.  In fact,
	\[
	p_L=\dim L_0(1-c)^{\dim\g_{1}}.
	\]
\end{lemma}
\begin{proof}
	The second statement clearly implies the first, and it follows from the fact that $L=\Ind_{\g_{0}\oplus\g_{1}}^{\g}L_0$ is a Kac-module (see for instance \cite{kac1977lie}).  Thus we have
	\[
	p_{L}(c)=\sum\limits_{i}(-1)^{i}c^i\dim L_0\begin{pmatrix}\dim\g_{1}\\ i\end{pmatrix}=\dim L_0(1-c)^{\dim\g_{1}}.
	\]
	However we may give another proof which generalizes to other situations.  If $L$ is projective of highest weight $\lambda$, then by Proposition \ref{mild_proj_criteria_z_full} we have $HC(\operatorname{ev}_{eK}v_{\g_c})(\lambda)\neq0$.  Thus $f_c\in L^*\boxtimes L$ must not vanish at $eG$, i.e. we must have $p_L(c)\neq0$.
\end{proof}

\begin{remark}
	It is now possible to define, for $c\in k^\times$ and a simple $G$-module $L$, the $c$-graded dimension of $L$ to be $p_L(c)$.  This definition also naturally arises if we consider type $I$ algebras as Lie algebra objects in the tensor category of $\Z$-graded vector spaces with the symmetric structure lifted from the category of super vector spaces.  Observe that $p_L(1)=\dim L$ and $p_L(-1)=\operatorname{sdim}(L)$.  
\end{remark}

\begin{remark}
	It would be interesting to understand the roots of $p_L$ for irreducible $L$, and in particular the order of vanishing at $c=1$, in terms of the representation theory of $L$.  For instance, $L$ is maximally atypical if and only if $p_L(1)\neq0$, while $L$ is projective if and only if the order of vanishing at 1 is $\dim\g_{-1}$. 
\end{remark}

\subsection{Limiting to the center} Write $\Aut:=\Aut(\g,\g_{\ol{0}})$, which is an algebraic group, and let $S\sub\Aut$ denote the subset of automorphisms without nonzero fixed vectors in $\g_{\ol{1}}$.  Then $S$ is open in $\Aut$, and further is nonempty because $\delta\in S$.  

If $\Aut$ has dimension bigger than 0 and $\Aut^0\cap S$ is nonempty, $\Aut^0\cap S$ will be Zariski dense.  Thus if we choose $(\phi_c)_{c\in k^\times}\sub S$ such that $\lim\limits_{c\to 0}\phi_c=\id_{\g}$, it is reasonable to consider, for $u\in \ZZ(\UU\g_{\ol{0}})$, the element 
\[
\lim\limits_{c\to 0}\ad_{\phi_c}(v_{\g})(u).
\]
This limit exists and is equal to $\ad(v_{\g})(u)\in \ZZ$.  However it is quite possibly zero, and in particular the above limit need not preserve Harish-Chandra polynomials.  A more fruitful approach is to choose elements (if they exist) $u_c\in \AA_{\phi_c}$ for each $c$ such that $HC(u_c)$ is constant.  Then we may consider the limit (if it exists):
\[
u_0:=\lim\limits_{c\to 0}u_c.
\]
If $u_0$ does exist then it must be in $\ZZ$ and have $HC(u_0)=HC(u_c)$ for all $c$.  However in general such a limit need not exist, for example in the case that $HC(\AA_c)=0$ for all $c$.  However for type I basic algebras this limit does exist, which we now prove.  Thus let $\g$ be a type $I$ basic Lie superalgebra, and let $N=\dim\g_{1}$.

\begin{definition}
Consider the filtration by degree on $S(\h)$, and pull this back under $HC$ to each $\AA_{\zeta_N^i}$ to obtain a filtration $K^\bullet\AA_{\zeta_N^i}$.  Then let $G^\bullet\ZZ_{full}$ be the filtration on $\ZZ_{full}$ given by
\[
G^n\ZZ_{full}:=\sum\limits_i K^n\AA_{\zeta_N^i}.
\]
This defines an algebra filtration on $\ZZ_{full}$ such that $G^n\ZZ_{full}$ is finite-dimensional for each $n$.
\end{definition} 

Now we have the following:

\begin{lemma}\label{inj_lemma}
	For each $n\in\N$, there exists typical integral dominant weights $\lambda_1,\dots,\lambda_s$ such that map $G^n\ZZ_{full}\to\End(L(\lambda_{1})\oplus\cdots\oplus L(\lambda_s))$ is injective.
\end{lemma}

\begin{proof}
	Choose a basis $p_1,\dots,p_k$ of $HC(K^n\AA_{-1})$, and extend this to a basis of $HC(K^n\ZZ)$, $p_1,\dots,p_k,p_{k+1},\dots,p_\ell$.  
	
	The integral typical dominant weights are Zariski dense in $\h^*$, so necessarily there exists typical integral dominant weights $\lambda_1,\dots,\lambda_s$ such that evaluation at these points induces an injective map on $HC(K^n\ZZ)$.  Now consider the map 
	\[
	\phi:G^n\ZZ_{full}\to\End(L(\lambda_1)\oplus\cdots\oplus L(\lambda_s)),
	\]
and suppose it is not injective.  We may write an arbitrary element in $G^n\ZZ_{full}$ as
	\[
	\sum\limits_{0\leq j\leq N-1}\sum\limits_{1\leq i\leq k}\alpha_{i,j}a_{i,j}+\sum\limits_{k<i\leq \ell}\beta_iz_i
	\]
	where $z_i\in K^n\ZZ$ and $a_{i,j}\in K^n\AA_{\zeta_N^j}$ such that $HC(z_i)=p_i$ and $HC(a_{i,j})=p_i$ for all valid $i$.  Now suppose that
	\[
	\sum\limits_{0\leq j\leq N-1}\sum\limits_{1\leq i\leq k}\alpha_{i,j}\phi(a_{i,j})+\sum\limits_{k<i\leq \ell}\beta_i\phi(z_i)=0
	\]
	Looking at the action on the highest weight vectors, this implies by our choice of $\lambda_1,\dots,\lambda_s$ that
	\[
	\sum\limits_{1\leq i\leq k,\zeta}\alpha_{i,j}p_i+\sum\limits_{k<i\leq \ell}\beta_ip_i=0.
	\]
	Thus we must have $\beta_i=0$ for all $i$, and $\sum\limits_{0\leq j\leq N-1}\alpha_{i,j}=0$ for all $i$. Looking further at the action on the $(-r)$th graded component of $L(\lambda_1)\oplus\cdots\oplus L(\lambda_s)$ according to the grading defined in Theorem \ref{HC_Z_full}, we find that
	\[
	\sum\limits_{i,j}\alpha_{i,j}\zeta^{rj}p_{i}=0
	\]
	which implies that $\sum\limits_{j}\alpha_{i,j}\zeta^{jr}=0$ for all $i,r$.  By the nonsingularity of the Vandermonde matrix, this implies $\alpha_{i,j}=0$, and we are done.  
\end{proof}

\begin{cor}\label{limit_to_center_thm}
	Choose an element $p\in HC(\AA_{-1})$, and let $a_{\lambda,p}\in\AA_{\lambda}$ satisfy $HC(a_{\lambda,p})=p$ for all $\lambda\in k^\times$.  Then as $\lambda\to 1$,	 $a_{\lambda,p}$ converges in $G^{2\deg p}\ZZ_{full}$ to the unique central element $z_p$ with $HC(z_p)=p$.
\end{cor}
\begin{proof}
	Choose an embedding $\phi:G^{2\deg p}\ZZ_{full}\to\End(L(\lambda_{1})\oplus\cdots\oplus L(\lambda_s))$ using Lemma \ref{inj_lemma}.  Now in $\End(L(\lambda_i))$ we have that $a_{\lambda,p}-z_p$ acts on $\Lambda^j\g_{-1}\otimes L_0(\lambda_i)$ as
	\[
	p(\lambda_i)(\lambda^i-1).
	\]
	Thus if we take $\lambda\to 1$ we get convergence as elements of $\End(L(\lambda_{1})\oplus\cdots\oplus L(\lambda_s))$.  Since $\phi$ is an injective linear map, we are done.
\end{proof}

\begin{remark}
It is now possible, in principle, to obtain explicit formulas for the elements of $\ZZ$ whose Harish-Chandra image lies in $HC(\AA_{-1})$.  In the case of $\g\l(1|1)$ for instance, this would give all elements of the center with trivial constant term.  For example, we may produce the known formula for the element of the center whose Harish-Chandra polynomial is a scalar multiple of
	\[
	t_\g=\prod\limits_{\alpha\in\Delta_1^+}(h_{\alpha}+(\alpha,\rho)).
	\]
	Let $u_1,\dots,u_N$ be a basis of $\g_{1}$ and $v_1,\dots,v_N$ a basis of $\g_{-1}$, and write $V=v_1\cdots v_N\in\UU\g$.  Then by Proposition \ref{v_g_for_type_I}, $v_{\g}=u_1\cdots u_Nv_1\cdots v_N$, and we see that
	\begin{eqnarray*}
	\ad_{c}(v_{\g})(1)& = &(1-c)^{N}\ad_{c}(u_1\cdots u_N)(V)\\
					  & = &(1-c)^{N}\sum\limits_{I\sub\{1,\dots,N\}}(-1)^{Nl+i_1+\dots+i_l} c^{-l}u_{I^c}V\tilde{u_{I}}.
	\end{eqnarray*}

Here the sum runs over all subsets $I$ of $\{1,\dots,N\}$, and we write $l=|I|$.  Here $I^c$ is the complement of $I$ as a set.  Further, we define for a subset $J=\{j_1<\dots<j_l\}\sub\{1,\dots,N\}$,
\[
v_{J}=v_{j_1}\cdots v_{j_l},\ \ \ \  \tilde{v_J}=v_{j_l}\cdots v_{j_1}.
\]
If we divide $\ad_{c}(v_{\g})(1)$ by $(1-c)^{N}$, the Harish-Chandra projection of the resulting element is constant and equal to $HC(u_1\cdots u_Nv_1\cdots v_N)$, which is $t_{\g}$ up to a constant.  Taking the limit $c\to 1$ we obtain the following element of the center:
\[
\sum\limits_{I\sub\{1,\dots,n\}}(-1)^{nl+i_1+\dots+i_l}u_{I^c}V\tilde{u_{I}}.
\]
In general the above process will not be as straightforward, as for a general $z_0\in\ZZ(\UU\g_{\ol{0}})$ we have that
\[
\ad_{c}(v_{\g})(z_0)=(1-c)^N\sum\limits_{I\sub\{1,\dots,n\}}(-1)^{nl+i_1+\dots+i_l} c^{-l}u_{I^c}Vz_0\tilde{u_{I}}+l.o.t
\]
where $l.o.t.$ denotes terms of lower order in $(1-c)$.  Thus we cannot divide by $(1-c)^N$.  An (albeit tedious) way to overcome this is to take $k[c,(1-c)^{-1}]$ linear combinations of elements $\ad_{c}(v_{\g})(z_0)$ in order to obtain an element in $\AA_{c}$ with constant Harish-Chandra polynomial.  For $\g\l(1|1)$ this could certainly be worked out, but for higher rank superalgebras the elements of the center of $\UU\g_{\ol{0}}$ are more complicated, making this process more challenging.

However we note that in \cite{gorelik2004kac} a method for algorithmically computing any element of $\ZZ$ with a given Harish-Chandra projection is given, based off an idea originally due to Kac.  
\end{remark}

\bibliographystyle{amsalpha}
\bibliography{bibliography}

\textsc{\footnotesize Dept. of Mathematics, Ben Gurion University, Beer-Sheva,	Israel} 

\textit{\footnotesize Email address:} \texttt{\footnotesize xandersherm@gmail.com}

\end{document}